\documentclass[abstract=on, 11pt, a4paper]{scrartcl}

\usepackage[T1]{fontenc}

\usepackage{amsmath, amssymb, amsthm}
\usepackage{mathtools}
\usepackage{latexsym}

\usepackage{amsbsy}

\usepackage[left=3cm, right=2.5cm, top=2.5cm, bottom=3.5cm]{geometry}

\usepackage{enumitem}
\usepackage{comment}

\usepackage{xcolor}

\usepackage{tikz-cd}
\usetikzlibrary{cd}

\usepackage{graphicx}
\usepackage{pdfpages}

\usepackage[super]{nth}

\usepackage{hyperref}
\usepackage[colorinlistoftodos]{todonotes}

\newcounter{countercheck}[section]

\theoremstyle{plain}
\swapnumbers
\newtheorem{theorem}[countercheck]{Theorem}
\newtheorem{proposition}[countercheck]{Proposition}
\newtheorem{lemma}[countercheck]{Lemma}
\newtheorem{corollary}[countercheck]{Corollary}

\theoremstyle{definition}

\newtheorem{convention}[countercheck]{Convention}
\newtheorem{example}[countercheck]{Example}

\theoremstyle{remark}
\newtheorem{remark}[countercheck]{Remark}

\renewcommand{\phi}{\varphi}
\renewcommand{\epsilon}{\varepsilon}

\newcommand{\NN}{\mathbb{N}}

\newcommand{\llangle}{\langle\langle}
\newcommand{\rrangle}{\rangle\rangle}

\newcommand{\Cp}{\mathcal{C}_{\epsilon}}
\newcommand{\Cm}{\mathcal{C}_{-\epsilon}}
\newcommand{\C}{\mathcal{C}}
\renewcommand{\P}{\mathcal{P}}
\newcommand{\Cbf}{\textbf{\textup{C}}}

\DeclareMathOperator{\Aut}{Aut}
\DeclareMathOperator{\Sym}{Sym}
\DeclareMathOperator{\proj}{proj}

\DeclareMathOperator{\Fix}{Fix}
\DeclareMathOperator{\Stab}{Stab}
\DeclareMathOperator{\id}{id}
\DeclareMathOperator{\co}{\textup{(co)}}
\DeclareMathOperator{\lco}{\textup{(lco)}}
\DeclareMathOperator{\lsco}{\textup{(lsco)}}

\numberwithin{equation}{section} 
\title{$3$-spherical twin buildings}
\author{Sebastian \textit{Bischof}\footnote{email: sebastian.bischof@math.uni-giessen.de} \\
	Mathematisches Institut, Arndtstra\ss e 2, 35392 Gie\ss en, Germany}

\date{\today}

\begin{document}


\maketitle

\begin{abstract}
	We classify thick irreducible $3$-spherical twin buildings of rank at least $3$ in which every panel contains at least $6$ chambers. Together with the Main result of \cite{Mu99} we obtain a classification of thick irreducible $3$-spherical twin buildings.
	
	
\end{abstract}

\section{Introduction}

In \cite{Ti74} Tits gave a complete classification of all thick irreducible spherical buildings of rank at least $3$. The decisive step in this classification is the extension theorem for isometries (Theorem $4.1.2$ in loc. cit.). It implies that a thick spherical building is uniquely determined by its local structure. Inspired by the paper \cite{Ti87} on Kac-Moody groups over fields, Ronan and Tits introduced twin buildings. These combinatorial objects appear to be natural generalisations of spherical buildings, as they come equipped with an opposition relation which shares many important properties with the opposition relation in a spherical building. In \cite{MR95} M\"uhlherr and Ronan gave a proof of the extension theorem for $2$-spherical twin buildings satisfying an additional condition they call $\co$. Condition $\co$ turns out to be rather mild (see \cite[Introduction]{MR95}). In order to classify such twin buildings it suffices to determine all possible local structures which appear in twin buildings.

The local structure of a building mentioned before is essentially the union of all rank $2$ residues of a chamber. In \cite{RT87} Ronan and Tits introduced the notion of a \textit{foundation} which is designed to axiomatize geometric structures that may occur as local structures of a building. Roughly speaking, it is a union of rank $2$ buildings which are glued along certain rank $1$ residues (for the precise definition we refer to Section \ref{Section: Foundations}). A foundation is called \textit{integrable}, if it is the local structure of a twin building. One can associate with each twin building of $2$-spherical type a foundation and all such are \textit{Moufang foundation}, i.e. the irreducible rank $2$ buildings are Moufang and the glueings are compatible with the Moufang structures induced on the rank $1$ residues.

Inspired by \cite{RT87}, Tits conjectured that a Moufang foundation is integrable if and only if each of its spherical rank $3$ restrictions is integrable (cf. \cite[Conjecture $2$]{Ti92}). It turned out that this conjecture was too optimistic and he gave a reformulation by omitting the word "spherical" in his earlier conjecture (cf. \cite{Ti73-00}). A proof of this conjecture would reduce the classification of $2$-spherical twin buildings to the rank $3$ case. In \cite{MHab} and \cite{MuSurvey} a strategy for the classification of $2$-spherical twin buildings is outlined that does not make use of this conjecture. Several important steps in this classification programme have been carried out in the meantime (cf. \cite{Mu99}, \cite{WeiDiss}, \cite{WenDiss}). The results obtained up until now suggest that the conjecture will be a consequence of the classification once it will be accomplished. However, it would be most desirable to have a proof of this conjecture that is independent of the classification. Our main result is a conceptual proof of this conjecture in the $3$-spherical case. It turned out that we only need the integrability for the restriction to irreducible types (cf. Corollary \ref{Corollary: Main result}):

\medskip
\noindent
\textbf{Theorem A:} Let $\mathcal{F}$ be a Moufang foundation of irreducible $3$-spherical type and of rank at least $3$ such that every panel contains at least $6$ chambers. Then the following are equivalent:
\begin{enumerate}[label=(\roman*)]
	\item $\mathcal{F}$ is integrable.
	
	\item Each irreducible rank $3$ restriction of $\mathcal{F}$ is integrable.
\end{enumerate}

\medskip A consequence of Theorem A together with \cite{Mu99,BM23} is the classification of thick irreducible $3$-spherical twin buildings (cf. Theorem \ref{Theorem: Classification result}):

\medskip
\noindent
\textbf{Corollary B:} Let $\Delta$ be a thick irreducible $3$-spherical twin building of rank at least $3$. Then $\Delta$ is known.

\medskip Let $(W, S)$ be a Coxeter system and let $\Phi$ be the associated set of roots (viewed as half-spaces). An \textit{RGD-system of type $(W, S)$} is a pair $(G, (U_{\alpha})_{\alpha \in \Phi})$ consisting of a group $G$ together with a family of subgroups $U_{\alpha}$ indexed by the set of roots satisfying a few axioms (for the precise definition see Section \ref{Section: Preliminaries}). Let $\mathcal{F}$ be a Moufang foundation of type $(W, S)$ satisfying a certain Condition $\lco$ (e.g. this is satisfied if every panel contains at least $5$ chambers). In Section \ref{Section: Foundations} we construct for all $s\neq t\in S$ RGD-systems $X_{s, t}$ acting on the corresponding building of the foundation. Moreover, we construct RGD-systems $X_s$ and canonical homomorphisms $X_s \to X_{s, t}$. Let $G$ be the direct limit of the groups $X_s$ and $X_{s, t}$. Then there is a natural way of defining a family of \textit{root groups} $(U_{\alpha})_{\alpha \in \Phi}$ inside $G$. Let $\mathcal{D}_{\mathcal{F}} := (G, (U_{\alpha})_{\alpha \in \Phi})$. We prove the following result (cf. Corollary \ref{inclusionsinjectivergd3}):

\medskip
\noindent
\textbf{Theorem C:} Let $\mathcal{F}$ be a Moufang foundation of $2$-spherical type satisfying Condition $\lco$. If the canonical mappings $U_{\pm \alpha_s} \to G$ are injective and if $\mathcal{D}_{\mathcal{F}}$ satisfies (RGD$3$), then $\mathcal{D}_{\mathcal{F}}$ is an RGD-system and $\mathcal{F}$ is integrated in the twin building associated to $\mathcal{D}_{\mathcal{F}}$.

\medskip We use Theorem C to prove Theorem A. Our strategy is to let $G$ act on a building and deduce that the hypotheses of Theorem C are satisfied. Thus we have a twin building $\Delta(\mathcal{D}_{\mathcal{F}})$ in which $\mathcal{F}$ is integrated. In particular, we have the following corollary (cf. Theorem \ref{Theorem: rank 3 intrgrable implies RGD-system}):

\medskip
\noindent
\textbf{Corollary D:} Let $\mathcal{F}$ be an irreducible Moufang foundation of $3$-spherical type such that every panel contains at least $6$ chambers. If each irreducible rank $3$ restriction of $\mathcal{F}$ is integrable, then $\mathcal{D}_{\mathcal{F}}$ is an RGD-system.

\medskip
\noindent
\textbf{Remarks:}

\smallskip
\noindent $1.$ In \cite{Mu99} M\"uhlherr accomplished the classification of thick locally finite twin buildings of irreducible $2$-spherical type and $m_{st} <8$ without residues associated with one of the groups $B_2(2), G_2(2), G_2(3), ^2F_4(2)$. In particular, the thick locally finite twin buildings of irreducible $3$-spherical type without residues associated with $B_2(2)$ are already known. As we will see in the proof of Theorem \ref{Theorem: Classification result}, the assumption about the residues can be dropped in the $3$-spherical case.

\smallskip
\noindent $2.$ By Corollary B we have a classification of $3$-spherical simply-laced (i.e. $m_{st} \in \{2, 3\}$) twin buildings. We note that in this case the integrability of those Moufang foundations is already established by different methods: In a $3$-spherical simply-laced Moufang foundation, the Moufang triangles are parametrised by skew-fields and all such are isomorphic. If there exists $s\in S$ with three neighbours in the Coxeter diagram, then there is a $D_4$ subdiagram (as there is no $\tilde{A}_2$ subdiagram) and hence the Moufang triangles are parametrised by a field. The existence of a twin building with the prescribed foundation follows now from \cite{Mu99}. If each $s\in S$ has at most $2$ neighbours, the diagram is either $A_n$ or $\tilde{A}_n$. In the first case the existence follows from projective geometry and in the second case from Kac-Moody theory.

\renewcommand{\abstractname}{Acknowledgement}
\begin{abstract}
	I am very grateful to Bernhard M\"uhlherr for proposing this project to me, as well as for many helpful comments and suggestions. 
\end{abstract}

\section{Preliminaries}\label{Section: Preliminaries}

\subsection*{Direct limits}

This subsection is based on \cite{Se79}.

Let $I$ be a set and let $(G_i)_{i \in I}$ be a family of groups. Furthermore, let $F_{i,j}$ be a set of homomorphisms of $G_i$ into $G_j$. Then a group $G := \lim\limits_{\longrightarrow} G_i$ together with a family of homomorphisms $f_i: G_i \to G$ such that $f_j \circ f = f_i$ holds for all $f\in F_{i, j}$ is called \textit{direct limit} of the groups $G_i$, relative to the $F_{i, j}$ if it satisfies the following condition:

If $H$ is a group and if $h_i: G_i \to H$ is a family of homomorphisms such that $h_j \circ f = h_i$ holds for all $f\in F_{i, j}$, then there exists exactly one homomorphism $h: G \to H$ such that $h_i = h \circ f_i$.

In this paper we only consider the case where $X_i$ and $X_{i, j} := \langle X_i, X_j \rangle$ are groups for $i \neq j \in I$ and $f: X_i \to X_{i, j}$ are the canonical homomorphisms. We call this direct limit the \textit{$2$-amalgam of the groups $X_i$}.

\subsection*{Coxeter systems}

Let $(W, S)$ be a Coxeter system and let $\ell$ denote the corresponding length function. For $s, t \in S$ we denote the order of $st$ in $W$ by $m_{st}$. The \textit{rank} of a Coxeter system is the cardinality of the set $S$. The \textit{Coxeter diagram} corresponding to $(W, S)$ is the labeled graph $(S, E(S))$, where $E(S) = \{ \{s, t \} \mid m_{st}>2 \}$ and where each edge $\{s,t\}$ is labeled by $m_{st}$ for all $s, t \in S$. We call a Coxeter system \textit{irreducible}, if the underlying graph is connected; otherwise we call it \textit{reducible}. It is well-known that the pair $(\langle J \rangle, J)$ is a Coxeter system (cf. \cite[Ch. IV, $\S 1$ Theorem $2$]{Bo68}). A subset $J \subseteq S$ is called \textit{spherical} if $\langle J \rangle$ is finite; for $k \in \NN$ the Coxeter system is called \textit{$k$-spherical}, if $\langle J \rangle$ is spherical for each subset $J$ of $S$ with $\vert J \vert \leq k$. Given a spherical subset $J$ of $S$, there exists a unique element of maximal length in $\langle J \rangle$, which we denote by $r_J$ (cf. \cite[Corollary $2.19$]{AB08}). For $s_1, \ldots, s_k \in S$ we call $s_1 \cdots s_k$ \textit{reduced}, if $\ell(s_1 \cdots s_k) = k$. Similar as in \cite[Ch. $2.2$]{Ro89} we define for distinct $s, t \in S$ with $m_{st} < \infty$ the element $p(s, t)$ to mean $stst \ldots$ with $\ell(p(s, t)) = m_{st}$; e.g. if $m_{st} = 3$, we have $p(s, t) = sts$. It is well-known that for $w\in W, s\in S$ with $\ell(sw) < \ell(w)$, there exists $w' \in W$ such that $\ell(w') = \ell(w) -1$ and $w = sw'$ (cf. \cite[exchange condition on p.$79$]{AB08}).

\begin{convention}
	For the rest of this paper we let $(W, S)$ be a Coxeter system.
\end{convention}

\subsection*{Chamber systems}

Let $I$ be a set. A \textit{chamber system over $I$} is a pair $\left(\C, (\sim_i)_{i \in I} \right)$ consisting of a non-empty set $\C$ whose elements are called \textit{chambers} and where $\sim_i$ is an equivalence relation on the set of chambers for each $i\in I$. Given $i\in I$ and $c, d \in \C$, then $c$ is called \textit{$i$-adjacent} to $d$ if $c \sim_i d$. The chambers $c, d$ are called \textit{adjacent}, if they are $i$-adjacent for some $i\in I$. If we restrict $\sim_i$ to $\emptyset \neq \mathcal{X} \subseteq \C$ for each $i\in I$, then $\left( \mathcal{X}, (\sim_i)_{i\in I} \right)$ is a chamber system over $I$ and we call it the \textit{induced} chamber system.

A \textit{gallery} in $(\C, (\sim_i)_{i\in I})$ is a sequence $(c_0, \ldots, c_k)$ such that $c_{\mu} \in \C$ for all $0 \leq \mu \leq k$ and $c_{\mu-1}, c_{\mu}$ are adjacent for all $1 \leq \mu \leq k$. Given a gallery $G = (c_0, \ldots, c_k)$, then we put $\beta(G) := c_0$ and $\epsilon(G) := c_k$. If $G$ is a gallery and if $c, d \in \C$ are such that $\beta(G) = c$ and $\epsilon(G) = d$ then we say that $G$ is a \textit{gallery from $c$ to $d$} or \textit{$G$ joins $c$ and $d$}. The chamber system $(\C, (\sim_i)_{i\in I})$ is said to be \textit{connected}, if for any two chambers there exists a gallery joining them. A gallery $G$ will be called \textit{closed} if $\beta(G) = \epsilon(G)$. Given two galleries $G = (c_0, \ldots, c_k)$ and $H = (d_0, \ldots, d_l)$ such that $\epsilon(G) = \beta(H)$, then $GH$ denotes the gallery $(c_0, \ldots, c_k = d_0, \ldots, d_l)$.

Let $J$ be a subset of $I$. A \textit{$J$-gallery} is a gallery $(c_0, \ldots, c_k)$ such that for each $1 \leq \mu \leq k$ there exists an index $j\in J$ with $c_{\mu -1} \sim_j c_{\mu}$.

\subsection*{Homotopy of galleries and simple connectedness}

In the context of chamber systems there is the notion of $m$-homotopy and $m$-simple connectedness for each $m \in \NN$. In this paper we are only concerned with the case $m=2$. Therefore our definitions are always to be understood as a specialisation of the general theory to the case $m=2$.

Let $(\C, (\sim_i)_{i\in I})$ be a chamber system over a set $I$. Two galleries $G$ and $H$ are said to be \textit{elementary homotopic} if there exist two galleries $X, Y$ and two $J$-galleries $G', H'$ for some $J \subseteq I$ of cardinality at most $2$ such that $G = XG'Y, H = XH'Y$. Two galleries $G, H$ are said to be \textit{homotopic} if there exists a finite sequence $G_0, \ldots, G_l$ of galleries such that $G_0 = G, G_l = H$ and such that $G_{\mu-1}$ is elementary homotopic to $G_{\mu}$ for all $1 \leq \mu \leq l$.

If two galleries $G, H$ are homotopic, then it follows by definition that $\beta(G) = \beta(H)$ and $\epsilon(G) = \epsilon(H)$. A closed gallery $G$ is said to be \textit{null-homotopic} if it is homotopic to the gallery $(\beta(G))$. The chamber system $(\C, (\sim_i)_{i\in I})$ is called \textit{simply connected} if it is connected and if each closed gallery is null-homotopic.

\subsection*{Buildings}

A \textit{building of type $(W, S)$} is a pair $\Delta = (\C, \delta)$ where $\C$ is a non-empty set and where $\delta: \C \times \C \to W$ is a \textit{distance function} satisfying the following axioms, where $x, y\in \C$ and $w = \delta(x, y)$:
\begin{enumerate}[label=(WD\arabic*), leftmargin=*]
	\item $w = 1_W$ if and only if $x=y$.
	
	\item if $z\in \C$ satisfies $s := \delta(y, z) \in S$, then $\delta(x, z) \in \{w, ws\}$, and if, furthermore, $\ell(ws) = \ell(w) +1$, then $\delta(x, z) = ws$.
	
	\item if $s\in S$, there exists $z\in \C$ such that $\delta(y, z) = s$ and $\delta(x, z) = ws$.
\end{enumerate}
The \textit{rank} of $\Delta$ is the rank of the underlying Coxeter system. The elements of $\C$ are called \textit{chambers}. Given $s\in S$ and $x, y \in \C$, then $x$ is called \textit{$s$-adjacent} to $y$, if $\delta(x, y) = s$. The chambers $x, y$ are called \textit{adjacent}, if they are $s$-adjacent for some $s\in S$. A \textit{gallery} from $x$ to $y$ is a sequence $(x = x_0, \ldots, x_k = y)$ such that $x_{l-1}$ and $x_l$ are adjacent for every $1 \leq l \leq k$; the number $k$ is called the \textit{length} of the gallery. Let $(x_0, \ldots, x_k)$ be a gallery and suppose $s_i \in S$ such that $\delta(x_{i-1}, x_i) = s_i$. Then $(s_1, \ldots, s_k)$ is called the \textit{type} of the gallery. A gallery from $x$ to $y$ of length $k$ is called \textit{minimal} if there is no gallery from $x$ to $y$ of length $<k$.

Given a subset $J \subseteq S$ and $x\in \C$, the \textit{$J$-residue of $x$} is the set $R_J(x) := \{y \in \C \mid \delta(x, y) \in \langle J \rangle \}$. A \textit{residue} is a subset $R$ of $\C$ such that there exist $J \subseteq S$ and $x\in \C$ with $R = R_J(x)$. Since the subset $J$ is uniquely determined by $R$, the set $J$ is called the \textit{type} of $R$ and the \textit{rank} of $R$ is defined to be the cardinality of $J$. Each $J$-residue is a building of type $(\langle J \rangle, J)$ with the distance function induced by $\delta$ (cf. \cite[Corollary $5.30$]{AB08}). Given $x\in \C$ and a $J$-residue $R \subseteq \C$, then there exists a unique chamber $z\in R$ such that $\ell(\delta(x, y)) = \ell(\delta(x, z)) + \ell(\delta(z, y))$ holds for every $y\in R$ (cf. \cite[Proposition $5.34$]{AB08}). The chamber $z$ is called the \textit{projection of $x$ onto $R$} and is denoted by $\proj_R x$. Moreover, if $z = \proj_R x$ we have $\delta(x, y) = \delta(x, z) \delta(z, y)$ for each $y\in R$. A residue is called \textit{spherical} if its type is a spherical subset of $S$. A building is \textit{spherical} if its type is spherical; for $k \in \NN$ it is called \textit{$k$-spherical} if $(W, S)$ is $k$-spherical. Let $R$ be a spherical $J$-residue. Then $x, y \in R$ are called \textit{opposite in $R$} if $\delta(x, y) = r_J$. A \textit{panel} is a residue of rank $1$. An \textit{$s$-panel} is a panel of type $\{s\}$ for $s\in S$. For $c \in \C$ and $s\in S$ we denote the $s$-panel containing $c$ by $\P_s(c)$. The building $\Delta$ is called \textit{thick}, if each panel of $\Delta$ contains at least three chambers. For an $s$-panel $P$ we will also write $P_s$ instead of $P$ to underline the type of $P$. We denote the set of all panels in a given building $\Delta$ by $\P_{\Delta}$. For $c\in \C$ and $k \in \NN$ we denote the union of all residues of rank at most $k$ containing $c$ by $E_k(c)$. Let $\Delta = (\C, \delta), \Delta' =(\C', \delta')$ be two buildings of type $(W, S)$ and let $X\subseteq \C, X' \subseteq \C'$. Then a map $\phi:X \to X'$ is called \textit{isomorphism} if it is bijective and preserves the distance functions. In this case we will write $X \cong X'$. We denote the set of all isomorphisms from a building $\Delta$ to itself by $\mathrm{Aut}(\Delta)$. 

\begin{remark}
	Our definition of a building agrees with the definition given in \cite[Ch. $3$]{Ro89} (cf. \cite[Proposition $5.23$]{AB08}).
\end{remark}

\subsection*{Coxeter buildings}

\begin{example}
	We define $\delta: W \times W \to W, (x, y) \mapsto x^{-1}y$. Then $\Sigma(W, S) := (W, \delta)$ is a building of type $(W, S)$. Moreover, $W$ acts on $\Sigma(W, S)$ by left-multiplication.
\end{example}

A \textit{reflection} is an element of $W$ that is conjugate to an element of $S$. For $s\in S$ we let $\alpha_s := \{ w\in W \mid \ell(sw) > \ell(w) \}$ be the \textit{simple root} corresponding to $s$. A \textit{root} is a subset $\alpha \subseteq W$ such that $\alpha = v\alpha_s$ for some $v\in W$ and $s\in S$. We denote the set of all roots by $\Phi(W, S)$. The set $\Phi(W, S)_+ = \{ \alpha \in \Phi(W, S) \mid 1_W \in \alpha \}$ is the set of all \textit{positive roots} and $\Phi(W, S)_- = \{ \alpha \in \Phi(W, S) \mid 1_W \notin \alpha \}$ is the set of all \textit{negative roots}. For $J \subseteq S$ we put $\Phi(W, S)^J := \Phi(\langle J \rangle, J)$ (resp. $\Phi(W, S)_+^J, \Phi(W, S)_-^J$). For each root $\alpha \in \Phi(W, S)$ we denote the \textit{opposite} root by $-\alpha$ and we let $r_{\alpha}$ be the unique reflection which interchanges these two roots. A pair of distinct roots $\{ \alpha, \beta \} \subseteq \Phi(W, S)$ is called \textit{prenilpotent} if both $\alpha \cap \beta$ and $(-\alpha) \cap (-\beta)$ are non-empty sets. Given such a pair $\{ \alpha, \beta \}$ we will write $\left[ \alpha, \beta \right] := \{ \gamma \in \Phi(W, S) \mid \alpha \cap \beta \subseteq \gamma \text{ and } (-\alpha) \cap (-\beta) \subseteq -\gamma \}$ and $(\alpha, \beta) := \left[ \alpha, \beta \right] \backslash \{ \alpha, \beta \}$. For a pair $\{ \alpha, \beta \} \subseteq \Phi$ of prenilpotent roots there are two possibilities: either $o(r_{\alpha} r_{\beta}) < \infty$ or $o(r_{\alpha} r_{\beta}) = \infty$. The second case implies $\alpha \subsetneq \beta$ or $\beta \subsetneq \alpha$.

\begin{convention}
	For the rest of this paper we let $(W, S)$ be a Coxeter system of finite rank and $\Phi := \Phi(W, S)$ (resp. $\Phi_+, \Phi_-$) be the set of roots (resp. positive roots, negative roots). 
\end{convention}

For $\alpha \in \Phi$ we denote by $\partial \alpha$ (resp. $\partial^2 \alpha$) the set of all panels (resp. spherical residues of rank $2$) stabilized by $r_{\alpha}$. The set $\partial \alpha$ is called the \textit{wall} associated to $\alpha$. For a gallery $G = (c_0, \ldots, c_k)$ we say that $G$ \textit{crosses the wall $\partial \alpha$} if $\{ c_{i-1}, c_i \} \in \partial \alpha$ for some $1 \leq i \leq k$.

\begin{lemma}\label{CM06Prop2.7}
	Let $\alpha \in \Phi$ and let $P, Q \in \partial \alpha$. Then there exist a sequences $P_0 = P, \ldots, P_n = Q$ of panels in $\partial \alpha$ and a sequence $R_1, \ldots, R_n$ of spherical rank $2$ residues in $\partial^2 \alpha$ such that $P_{i-1}, P_i$ are distinct and contained in $R_i$.
\end{lemma}
\begin{proof}
	This is a consequence of \cite[Proposition $2.7$]{CM06}. The fact that $P_0, \ldots, P_n \in \partial \alpha$ follows from the implication (iii)$\Rightarrow$(ii) in loc.cit. Since $P_i \subseteq R_i$, we infer $R_i \in \partial^2 \alpha$.
\end{proof}

Let $\Delta = (\C, \delta)$ be a building of type $(W, S)$. A subset $\Sigma \subseteq \C$ is called \textit{apartment} if it is isomorphic to $W$. A subset $\alpha \subseteq \C$ is called a \textit{root} if it is isomorphic to $\alpha_s$ for some $s\in S$.

\subsection*{Two conditions for buildings}

\begin{example}
	Let $\Delta = (\C, \delta)$ be a building of type $(W, S)$. We define $x \sim_s y$ if $\delta(x,y) \in \langle s \rangle$. Then $\sim_s$ is an equivalence relation and $(\C, (\sim_s)_{s\in S})$ is a chamber system.
\end{example}

We now introduce two conditions for a building:

\begin{enumerate}[leftmargin=*, align=left, itemindent=!]
	\item[$\lco$] A building $\Delta$ satisfies Condition $\lco$ if it is $2$-spherical and if $R$ is a rank $2$ residue of $\Delta$ containing a chamber $c$, then the induced chamber system defined by the set of chambers opposite $c$ inside $R$ is connected.

	\item[$\lsco$] A building $\Delta$ satisfies Condition $\lsco$ if it is $3$-spherical and if $R$ is a rank $3$ residue of $\Delta$ containing a chamber $c$, then the induced chamber system defined by the set of chambers opposite $c$ inside $R$ is simply connected.
\end{enumerate}

Buildings which satisfy both conditions are discussed in \cite[Chapter $9$]{DMVM11} and \cite[Introduction]{HNVM16}. In \cite{DMVM11} it is stated that a $3$-spherical building of simply-laced type (i.e. $m_{st} \in \{2, 3\}$ for all $s, t \in S$) in which every panel contains at least $4$ chambers satisfies both conditions $\lco$ and $\lsco$. Moreover, (following \cite{DMVM11} and \cite{HNVM16}) a (general) $3$-spherical building satisfies the Conditions $\lco$ and $\lsco$ if every panel contains at least $6$ chambers.

\subsection*{Spherical Moufang buildings}

Let $\Delta = (\C, \delta)$ be a thick irreducible spherical building of type $(W, S)$ and of rank at least $2$. For a root $\alpha$ of $\Delta$ we define the \textit{root group} $U_{\alpha}$ as the set of all automorphisms fixing $\alpha$ pointwise and fixing every panel $P$ pointwise, where $\vert P \cap \alpha \vert = 2$. The building $\Delta$ is called \textit{Moufang} if for every root $\alpha$ of $\Delta$ the root group $U_{\alpha}$ acts simply transitive on the set of apartments containing $\alpha$.

Let $\Delta = (\C, \delta)$ be a Moufang building of type $(W, S)$, let $\Sigma$ be an apartment of $\Delta$ and let $c\in \Sigma$. Identify the set of roots in $\Sigma$ with $\Phi$ and let $G = \langle U_{\alpha} \mid \alpha \in \Phi \rangle$. Let $H = \Fix_G(\Sigma)$ and $B_+ = \Stab_G(c) = H \langle U_{\alpha} \mid c \in \alpha \in \Phi \rangle$. Then $\Delta(G, B_+) = (G/B_+, \delta)$ is a building of type $(W, S)$, where $\delta: G/B_+ \times G/B_+ \to W, (gB_+, hB_+) \mapsto w$, where $g^{-1}h \in B_+ w B_+$ (for more information we refer to \cite[Section $6,7$]{AB08}). Moreover, the buildings $\Delta$ and $\Delta(G, B_+)$ are isomorphic. By \cite[Corollary $7.68$]{AB08} and \cite[Lemma $(7.4)$]{Ro89} every chamber $gB_+$ can be written in the form $u_1 n_{s_1} \cdots u_k n_{s_k}B_+$ with $u_i \in U_{\alpha_{s_i}}$ and $s_1 \cdots s_k = \delta(B_+, gB_+)$ is reduced. If we fix the \textit{type} $(s_1, \ldots, s_k)$ of $u_1n_{s_1} \cdots u_kn_{s_k}B_+$, then the elements $u_i \in U_{\alpha_{s_i}}$ are uniquely determined.

\subsection*{Twin buildings}

Let $\Delta_+ = (\C_+, \delta_+), \Delta_- = (\C_-, \delta_-)$ be two buildings of the same type $(W, S)$. A \textit{codistance} (or \textit{twinning}) between $\Delta_+$ and $\Delta_-$ is a mapping $\delta_*: (\C_+ \times \C_-) \cup (\C_- \times \C_+) \to W$ satisfying the following axioms, where $\epsilon \in \{+,-\}, x\in \Cp, y\in \Cm$ and $w = \delta_*(x, y)$:
\begin{enumerate}[label=(Tw\arabic*), leftmargin=*]
	\item $\delta_*(y, x) = w^{-1}$;
	
	\item if $z\in \Cm$ is such that $s := \delta_{-\epsilon}(y, z) \in S$ and $\ell(ws) = \ell(w) -1$, then $\delta_*(x, z) = ws$;
	
	\item if $s\in S$, there exists $z\in \C_{-\epsilon}$ such that $\delta_{-\epsilon}(y, z) = s$ and $\delta_*(x, z) = ws$.
\end{enumerate}
A \textit{twin building of type $(W, S)$} is a triple $\Delta = (\Delta_+, \Delta_-, \delta_*)$, where $\Delta_+ = (\C_+, \delta_+), \Delta_- = (\C_-, \delta_-)$ are buildings of type $(W, S)$ and where $\delta_*$ is a twinning between $\Delta_+$ and $\Delta_-$. The twin building is called \textit{thick}, if $\Delta_+$ and $\Delta_-$ are thick. The twin building $\Delta$ satisfies Condition $\lco$ if both buildings $\Delta_+, \Delta_-$ satisfy Condition $\lco$. For $\epsilon \in \{+,-\}$ and $c\in \C_{\epsilon}$ we define $c^{\mathrm{op}} = \{ d \in \C_{-\epsilon} \mid \delta_*(c, d) = 1_W \}$. Let $\Sigma_+ \subseteq \C_+$ and $\Sigma_- \subseteq \C_-$ be apartments of $\Delta_+$ and $\Delta_-$, respectively. Then the set $\Sigma := \Sigma_+ \cup \Sigma_-$ is called \textit{twin apartment} if for all $\epsilon \in \{+, -\}$ and $x \in \Sigma_{\epsilon}$ there exists exactly one $y\in \Sigma_{-\epsilon}$ with $\delta_*(x, y) = 1_W$. Moreover, for any $c_+ \in \C_+, c_- \in \C_-$ with $\delta_*(c_+, c_-) = 1_W$ there exists a unique twin apartment $\Sigma$ containing $c_+$ and $c_-$ (cf. \cite[Proposition $5.179(1)$]{AB08}). Let $\Delta' = (\Delta_+', \Delta_-', \delta_*')$ be another twin building of type $(W, S)$. A mapping $\phi: \C_+ \cup \C_- \to \C_+' \cup \C_-'$ is called \textit{isomorphism}, if it preserves the sign, the distance and the codistance.

\begin{lemma}\label{E1thick}
	Let $\Delta = (\Delta_+, \Delta_-, \delta_*)$ be a twin building of type $(W, S)$ and let $c$ be a chamber of $\Delta$ such that $\vert \P_s(c) \vert \geq 3$ holds for all $s\in S$. Then $\Delta$ is thick.
\end{lemma}
\begin{proof}
	Let $\epsilon \in \{+, -\}$ and let $c\in \C_{\epsilon}$. For each $d\in c^{\mathrm{op}}$ and each $s\in S$ the $s$-panel containing $d$ contains at least three elements by \cite[Corollary $5.153$]{AB08}. Let $c' \in \C_{\epsilon}$ and $d \in c^{\mathrm{op}}$. If $c' \in d^{\mathrm{op}}$ the $s$-panel containing $c'$ contains at least three chambers by \cite[Corollary $5.153$]{AB08}.	Otherwise, there exists a chamber $d' \in E_1(d)$ such that $d' \in c^{\mathrm{op}}$ and $\ell(\delta(c', d')) = \ell(\delta(c', d)) -1$. Using induction we obtain a chamber $d'' \in \C_{-\epsilon}$ such that $d'' \in c^{\mathrm{op}} \cap (c')^{\mathrm{op}}$. Applying \cite[Corollary $5.153$]{AB08} twice, we obtain that the $s$-panel containing $c'$ contains at least three chambers for each $s\in S$. Thus $\Delta_{\epsilon}$ is a thick building. The thickness of $\Delta_{-\epsilon}$ follows similarly.
\end{proof}

\begin{lemma}\label{Lemma: Apartment of foundation is contained in Apartment of twin building}
	Let $\Delta = (\Delta_+, \Delta_-, \delta_*)$ be a thick $2$-spherical twin building of type $(W, S)$, let $\epsilon \in \{+, -\}, c\in \C_{\epsilon}$ and for each $J \subseteq S$ with $\vert J \vert = 2$ we let $\Sigma_J$ be an apartment of $R_J(c)$ containing $c$ and such that $\Sigma_{\{r, s\}} \cap \P_s(c) = \Sigma_{\{s, t\}} \cap \P_s(c)$ for all $r\neq s \neq t \in S$. Then there exists a twin apartment $\Sigma$ such that $\Sigma_J \subseteq \Sigma$.
\end{lemma}
\begin{proof}
	This follows from \cite[Theorem $6.3.6$]{WenDiss}.
\end{proof}

\subsection*{RGD-systems}

An \textit{RGD-system of type $(W, S)$} is a pair $\left( G, \left( U_{\alpha} \right)_{\alpha \in \Phi}\right)$ consisting of a group $G$ together with a family of subgroups $U_{\alpha}$ (called \textit{root groups}) indexed by the set of roots $\Phi$, which satisfies the following axioms, where $H := \bigcap_{\alpha \in \Phi} N_G(U_{\alpha})$ and $U_{\epsilon} := \langle U_{\alpha} \mid \alpha \in \Phi_{\epsilon} \rangle$ for $\epsilon \in \{+, -\}$:
\begin{enumerate}[label=(RGD\arabic*), leftmargin=*] \setcounter{enumi}{-1}
	\item For each $\alpha \in \Phi$, we have $U_{\alpha} \neq \{1\}$.
	
	\item For each prenilpotent pair $\{ \alpha, \beta \} \subseteq \Phi$, the commutator group $[U_{\alpha}, U_{\beta}]$ is contained in the group $U_{(\alpha, \beta)} := \langle U_{\gamma} \mid \gamma \in (\alpha, \beta) \rangle$.
	
	\item For each $s\in S$ and each $u\in U_{\alpha_s} \backslash \{1\}$, there exist $u', u'' \in U_{-\alpha_s}$ such that the product $m(u) := u' u u''$ conjugates $U_{\beta}$ onto $U_{s\beta}$ for each $\beta \in \Phi$.
	
	\item For each $s\in S$, the group $U_{-\alpha_s}$ is not contained in $U_+$.
	
	\item $G = H \langle U_{\alpha} \mid \alpha \in \Phi \rangle$.
\end{enumerate}

Let $\mathcal{D} = (G, (U_{\alpha})_{\alpha \in \Phi})$ be an RGD-system of type $(W, S)$ and let $H := \bigcap_{\alpha \in \Phi} N_G(U_{\alpha})$ and $B_{\epsilon} := H \langle U_{\alpha} \mid \alpha \in \Phi_{\epsilon} \rangle$ for $\epsilon \in \{+, -\}$. It follows from \cite[Theorem $8.80$]{AB08} that there exists an \textit{associated} twin building $\Delta(\mathcal{D}) = (\Delta_+, \Delta_-, \delta_*)$ of type $(W, S)$ such that $\Delta_+ = (G/B_+, \delta_+)$ and $\Delta_- = (G/B_-, \delta_-)$ on which $G$ acts via left multiplication.

We define $X_s := \langle U_{\alpha_s} \cup U_{-\alpha_s} \rangle$ and $X_{s, t} := \langle X_s \cup X_t \rangle$ for all $s \neq t \in S$. If $(W, S)$ is $2$-spherical, $G = \langle U_{\alpha} \mid \alpha \in \Phi \rangle$ and $\Delta(\mathcal{D})$ satisfies Condition $\lco$, it follows by the Main result of \cite{AM97} that $G$ is isomorphic to the direct limit of the inductive system formed by the groups $L_s := HX_s$ and $L_{s, t} := HX_{s, t}$ for all $s\neq t \in S$. Furthermore, the direct limit of the inductive system formed by the groups $X_s$ and $X_{s, t}$ ($s \neq t \in S$) can be naturally endowed with an RGD-system and is a central extension of $G$ (cf. \cite[Theorem $3.7$]{Ca06}).

For $s\in S$ we define $H_s := \langle m(u) m(v) \mid u, v \in U_{\alpha_s} \backslash \{1\} \rangle$. By \cite[Lemma $3.3$]{Ca06} we have $H = \prod_{s\in S} H_s$, if $G = \langle U_{\alpha} \mid \alpha \in \Phi \rangle$ and if $(W, S)$ is $2$-spherical. By \cite[Consequence $(6)$ on p.$415$]{AB08} we have $H_t^{m(u)} \subseteq H_s H_t$ for all $1 \neq u \in U_{\alpha_s}$. For the rest of this subsection we fix $1\neq e_s \in U_{\alpha_s}$ and put $n_s := m(e_s)$. By \cite[Consequence $(11)$ on p.$416$]{AB08} there exist for every $1 \neq u_s \in U_{\alpha_s}$ elements $\overline{u}_s \in U_{\alpha_s}$ and $b(u_s) \in \langle U_{\alpha_s} \cup H_s \rangle$ such that $n_s u_s n_s = \overline{u}_s n_s b(u_s)$ (note that $\langle U_{\alpha_s} \cup H_s \rangle \cap U_{-\alpha_s} = \{1\}$). As in the case of a Coxeter system, we define $p(n_s, n_t)$ to mean $n_sn_tn_sn_t \ldots$, where $n_s, n_t$ appear $m_{st}$ times, e.g. if $m_{st} = 3$, we have $p(n_s, n_t) = n_s n_t n_s$. By \cite[Lemma $7.3$]{Ro89} we have $p(n_s, n_t) = p(n_t, n_s)$.

\begin{theorem}\label{presentationofanRGDsystem}
	Let $\left( G, (U_{\alpha})_{\alpha \in \Phi} \right)$ be an RGD-system spherical type $(W, S)$ such that $G = \langle U_{\alpha} \mid \alpha \in \Phi \rangle$. Let $B :=\langle U_+, H \rangle$. Then $G$ has the following presentation: as generators we have $\bigcup_{s\in S} H_s \cup \bigcup_{\alpha \in \Phi_+} U_{\alpha}$ and $\{ n_s \mid s\in S \}$ and as relations we have all relations in $B$ and for $s, t \in S, \alpha_s \neq \alpha \in \Phi_+, h \in H_t, u_s \in U_s$ we have the relations
	\begin{align*}
		& n_s h n_s = n_s^2 h^{n_s}, && n_s u_{\alpha} n_s = n_s^2 u_{\alpha}^{n_s}, && n_s u_s n_s = \overline{u}_s n_s b(u_s), && p(n_s, n_t) = p(n_t, n_s),
	\end{align*}
	where $u_{\alpha}^{n_s} \in U_{s\alpha}, n_s^2 \in H_s, h^{n_s} \in H_s H_t$.
\end{theorem}
\begin{proof}
	All these relations are relations in $G$. Thus it suffices to show that all relations in \cite[Corollary $(13.4)$]{Ti74} can be deduced from the relations in the statement. A case distinction yields the result.
\end{proof}

\begin{example}
	Let $\Delta$ be an irreducible spherical Moufang building of rank at least $2$ and let $\Sigma$ be an apartment of $\Delta$. Identifying the set of roots $\Phi$ with the set of roots in $\Sigma$, we deduce that $G = \langle U_{\alpha} \mid \alpha \in \Phi \rangle$ is an RGD-system.
\end{example}

\begin{lemma}\label{coUplus}
	Let $\left( G, \left(U_{\alpha} \right)_{\alpha \in \Phi} \right)$ be an RGD-system of $2$-spherical type $(W, S)$ such that $\Delta(\mathcal{D})$ satisfies Condition $\lco$ (e.g. each root group contains at least four elements). Then we have $\langle U_{\gamma} \mid \gamma \in [\alpha_s, \alpha_t] \rangle = \langle U_{\alpha_s} \cup U_{\alpha_t} \rangle$ or equivalently $[U_{\alpha_s}, U_{\alpha_t}] = U_{(\alpha_s, \alpha_t)}$ for all $s\neq t \in S$.
\end{lemma}
\begin{proof}
	This is a consequence of \cite[Lemma $18$ and Proposition $7$]{Ab96}.
\end{proof}

\begin{remark}\label{Remark: Homomorphism between RGD-systems}
	Let $\left( G, \left( U_{\alpha} \right)_{\alpha \in \Phi} \right), \left( H, \left( V_{\alpha} \right)_{\alpha \in \Phi} \right)$ be two RGD-systems of the same irreducible spherical type $(W, S)$ and of rank at least $2$. Assume that $G = \langle U_{\alpha} \mid \alpha \in \Phi \rangle$, the root groups $U_{\alpha}$ are nilpotent and that both twin buildings associated with the RGD-systems satisfy Condition $\lco$. Then $G, H$ are perfect by the previous lemma. Assume that there exists a homomorphism $\phi: G \to H$ such that $\phi(U_{\alpha}) = V_{\alpha}$. As $\ker \phi \neq G$, we have $\ker \phi \leq Z(G)$ by \cite[Proposition $7.132$]{AB08}.
\end{remark}

\begin{lemma}\label{Uplus2amalgam}
	Let $\mathcal{D} = \left( G, \left(U_{\alpha} \right)_{\alpha \in \Phi} \right)$ be an RGD-system of $3$-spherical type $(W, S)$ such that $\Delta(\mathcal{D})_{\pm}$ satisfy the Conditions $\lco$ and $\lsco$ (e.g. every root group contains at least five elements). Then $B_+$ is the $2$-amalgam of the groups $H U_{\alpha_s}$.
\end{lemma}
\begin{proof}
	This is a consequence of \cite[Corollary $1.2$]{DM07}.
\end{proof}

\begin{lemma}\label{Lemma: RGD-system residue isomorphism}
	Let $(G, (U_{\alpha})_{\alpha \in \Phi})$ be an RGD-system of type $(W, S)$ and let $X_{s, t} = \langle U_{\alpha} \mid \alpha \in \Phi^{\{s, t\}} \rangle, B_{\{s, t\}} = \langle H_s \cup H_t \cup U_{\alpha} \mid \alpha \in \Phi_+^{\{s, t\}} \rangle$. Then $(X_{s, t}, (U_{\alpha})_{\alpha \in \Phi^{\{s, t\}}})$ is an RGD-system and $X_{s, t} / B_{s, t} \to R_{\{s, t\}}(B_+), gB_{s, t} \mapsto gB_+$ is an isomorphism of buildings.
\end{lemma}
\begin{proof}
	As $X_{s, t}$ stabilizes $R_{\{s, t\}}(B_+)$, we deduce $\{ gB_+ \mid g \in X_{s, t} \} \subseteq R_{\{s, t\}}(B_+)$. Let $g \in G$ be such that $w := \delta_+(B_+, gB_+) \in \langle s, t \rangle$. By \cite[Proposition $8.59(2)$]{AB08}) there exists $h\in U_w$ such that $hwB_+ = gB_+$ and hence $gB_+ \in \{ hB_+ \mid h \in X_{s, t} \}$. Thus $R_{\{s, t\}}(B_+) = \{ gB_+ \mid g\in X_{s, t} \}$. Since $B_{\{s, t\}} \subseteq B_+$, the map is well-defined. Suppose $g, h\in X_{s, t}$ such that $gB_+ = hB_+$. Then $g^{-1}h \in B_+ \cap X_{s, t} = B_{\{s, t\}}$ as $B_+ = \Stab_G(B_+)$ and $B_{\{s, t\}} = \Stab_{X_{s, t}}(B_+)$ and thus the mapping is injective. Since $R_{\{s, t\}}(B_+) = \{ gB_+ \mid g\in X_{s, t} \}$ the mapping is also surjective. Furthermore, it preserves adjacency. The claim follows now from \cite[Lemma $5.61$]{AB08}.
\end{proof}

\subsection*{Blueprints}

This subsection is based on \cite[Ch. $7$]{Ro89}.

A \textit{parameter system} will mean a family of disjoint \textit{parameter sets} $(U_s')_{s\in S}$, each having a distinguished element $\infty_s \in U_s'$. We shall write $U_s := U_s' \backslash \{ \infty_s \}$.

Let $\Delta = (\C, \delta)$ be a building of type $(W, S)$ and let $(U_s')_{s\in S}$ be a parameter system. A \textit{labelling of $\Delta$ of type $(U_s')_{s\in S}$}, \textit{based at $c\in \C$}, is a family $\left( \phi_{P_s}: U_s' \to P_s \right)_{P_s \in \P_{\Delta}}$ of bijections such that $\phi_{P_s}(\infty_s) = \proj_{P_s} c$. We call $\phi_{P_s}^{-1}(x) \in U_s'$ the \textit{$s$-label} of $x$.

\begin{example}
	Let $\Delta = (\C, \delta)$ be a building of type $I_2(m)$, $m\geq 2$, with a labelling of type $(U_s', U_t')$, where $S = \{s, t\}$, based at some chamber $c\in \C$. Given any chamber $x\in \C$ at distance $d$ from $c$ one has a minimal gallery $(c = c_0, \ldots, c_d = x)$. Let $u_i$ be the label attached to $c_i$ in the unique panel containing $c_{i-1}$ and $c_i$. The gallery thus determines the sequence $(u_1, \ldots, u_d)$ where the $u_i$ lie alternately in $U_s$ and $U_t$, and any such sequence obviously determines a unique gallery starting at $c$, and hence a unique chamber at the end of this gallery. If two sequences determine the same chamber they are called \textit{equivalent}.
\end{example}

A \textit{blueprint of type $(W, S)$} is a tuple $\left( \left( U_s' \right)_{s\in S}, \left(\Delta_{st} \right)_{s\neq t\in S} \right)$ consisting of a parameter system $(U_s')_{s\in S}$ and buildings $\Delta_{st} = (\C_{st}, \delta_{st})$ of type $I_2(m_{st})$ for each $s\neq t \in S$, with a given labelling of type $(U_s', U_t')$, based at some chamber $\infty_{st} \in \C_{st}$.

A building $\Delta$ of type $(W, S)$ will be said to \textit{conform} to the blueprint $\left( \left( U_s' \right)_{s\in S}, \left( \Delta_{st} \right)_{s\neq t \in S} \right)$ if there exists a labelling of $\Delta$ of type $\left( U_s' \right)_{s\in S}$, based at some chamber $c\in \C$, such that for each $\{s, t\}$-residue $R$ of $\Delta$ there is an isomorphism $\phi_R: \Delta_{st} \to R$ with the property that $x$ and $\phi_R(x)$ have the same $s$- and $t$-labels for each chamber $x$ of $\Delta_{st}$. We call a blueprint $\mathcal{B}$ \textit{realisable} if there exists a building which conforms to it.

\subsection*{A realisable criterion}

We want to construct a building which conforms to a given blueprint. Let $\mathcal{B} = \left( \left( U_s' \right)_{s\in S}, \left( \Delta_{st} \right)_{s\neq t \in S} \right)$ be a blueprint of type $(W, S)$. As a first step we construct a chamber system $\Cbf(\mathcal{B})$ as follows: The chambers are sequences $\bar{u} := (u_1, \ldots, u_k)$, where $u_i \in U_{s_i}$ and $s_1 \cdots s_k$ is reduced. We call $(s_1, \ldots, s_k)$ the \textit{type} of $\bar{u}$. We define $s$-adjacency via
\begin{align*}
	(u_1, \ldots, u_k) \sim_s (u_1, \ldots, u_k, u_{k+1}) \sim_s (u_1, \ldots, u_k, u_{k+1}'),
\end{align*}
where $u_{k+1}, u_{k+1}' \in U_s$; this is evidently an equivalence relation, so $\Cbf(\mathcal{B})$ is a chamber system. For $\bar{u} = (u_1, \ldots, u_k)$ having type $(s_1, \ldots, s_k)$ and $\bar{v} = (v_1, \ldots, v_n)$ having type $(t_1, \ldots, t_n)$ we define the sequence $\bar{u} \bar{v} := (u_1, \ldots, u_k, v_1, \ldots, v_n)$ if $s_1 \cdots s_k t_1 \cdots t_n$ is reduced. We now define an \textit{elementary equivalence} to be an alteration from a sequence $\bar{u}_1 \bar{u} \bar{u}_2$ of type $(f_1, p(s, t), f_2)$ to $\bar{u}_1 \bar{u}' \bar{u}_2$ of type $(f_1, p(t, s), f_2)$ where $\bar{u}$ and $\bar{u}'$ are equivalent in $\Delta_{st}$. Two sequences $\bar{u}$ and $\bar{v}$ are called \textit{equivalent}, written $\bar{u} \simeq \bar{v}$, if one can be transformed to the other by a finite sequence of elementary equivalences. We now consider $\Cbf(\mathcal{B})$ modulo the equivalence relation. Notice that $[\bar{u}]$ determines a unique element $w \in W$ where $(s_1, \ldots, s_k)$ is the type of $\bar{u}$ and $w = s_1 \cdots s_k$. We define $x \sim_s y$ if $x = [\bar{u}], y = [\bar{v}]$ with $\bar{u} \sim_s \bar{v}$.

\begin{theorem}\label{Ro89Theorem7.1+7.2}
	Assume that for any two sequences $\bar{u},\bar{v}$ of the same reduced type, $\bar{u} \simeq \bar{v}$ implies $\bar{u} = \bar{v}$. Then $\Cbf_{\mathcal{B}} := ( [\bar{u}], (\sim_s)_{s\in S} )$ is a chamber system. The chambers are equivalence classes of sequences $\bar{u} := (u_1, \ldots, u_k)$, denoted by $[\bar{u}]$ or $[u_1, \ldots, u_k]$. Moreover, $\Cbf_{\mathcal{B}}$ is a building which conforms to $\mathcal{B}$.
\end{theorem}
\begin{proof}
	This follows from the proof of \cite[Theorem $(7.1)$]{Ro89}.
\end{proof}

\begin{corollary}\label{Corollary: Ro89Theorem7.1+7.2}
	A blueprint is realisable if and only if for any two sequences $\bar{u},\bar{v}$ of the same reduced type, $\bar{u} \simeq \bar{v}$ implies $\bar{u} = \bar{v}$. In particular, a blueprint is realisable if its restriction to each spherical rank $3$ subdiagram is realisable.
\end{corollary}
\begin{proof}
	The first assertion is \cite[Theorem $(7.2)$]{Ro89}; the second assertion follows from the first and \cite[Step $1$ of Theorem $(7.1)$]{Ro89}.
\end{proof}

\subsection*{Blueprints and Moufang buildings}

This subsection is based on \cite{Ro89}.

Let $\Delta = (\C, \delta)$ be a spherical Moufang building of type $(W, S)$ and let $c\in \C$. For each $s\in S$ we fix $1 \neq e_s \in U_{\alpha_s}$ and put $n_s := m(e_s)$. Then any chamber of $\Delta$ can be written uniquely in the form $u_1 n_{s_1} \cdots u_k n_{s_k}B_+$ with $u_i \in U_{\alpha_{s_i}}$ and $s_1 \cdots s_k$ is reduced, if we fix the type $(s_1, \ldots, s_k)$ of $u_1n_{s_1} \cdots u_kn_{s_k}B_+$. This yields a \textit{natural labelling} of the building $\Delta$. More precisely let $P$ be any $s$-panel of $\Delta$, and let $\proj_P c = d$ and $w = \delta(c, d)$. As cosets of $B_+$ the chambers of $P$ may be written $u_1n_{s_1} \cdots u_kn_{s_k}B_+$ (this is $d$), and $u_1n_{s_1} \cdots u_kn_{s_k} vn_sB_+$ where $u_i \in U_{\alpha_{s_i}}$ and $v\in U_{\alpha_s}$. We assign them the $s$-labels $\infty_s$ and $v$, using $U_{\alpha_s}' = U_{\alpha_s} \cup \{ \infty_s \}$. If we let $R_{st}$ be the $\{s, t\}$-residue containing $c$, then $R_{st}$ acquires a labelling and we have a blueprint given by the $\left( e_s \right)_{s\in S}$, namely $\left( (U_{\alpha_s}')_{s\in S}, (R_{st})_{s \neq t \in S} \right)$.

\begin{proposition}\label{Ro89Proposition7.5}
	Let $\Delta$ be a spherical Moufang building. Then $\Delta$ conforms to the blueprint given by the restriction to $E_2(c)$, i.e. $\left( (U_{\alpha_s}')_{s\in S}, (R_{st})_{s \neq t \in S} \right)$, and the natural labelling of $\Delta$ as above.
\end{proposition}
\begin{proof}
	This is \cite[$(7.5)$ Proposition]{Ro89}.
\end{proof}

\begin{corollary}\label{Homotopyinrank2}
	Let $\Delta$ be a spherical Moufang building and $\mathcal{B}$ be the blueprint given by $E_2(c)$. Then $(u_1, \ldots, u_k)$ and $(v_1, \ldots, v_k)$ are equivalent if and only if $u_1n_1 \cdots u_kn_k B_+ = v_1n_1' \cdots v_kn_k' B_+$. In particular, the map $\phi: \Cbf_{\mathcal{B}} \to \Delta, [(u_1, \ldots, u_k)] \mapsto u_1n_1 \cdots u_kn_k B_+$ is an isomorphism of buildings.
\end{corollary}
\begin{proof}
	If $(u_1, \ldots, u_k)$ and $(v_1, \ldots, v_k)$ are elementary equivalent, then $u_1n_1 \cdots u_kn_k B_+ = v_1n_1' \cdots v_kn_k' B_+$. Thus $\phi$ is well-defined. Clearly, $\phi$ is surjective. Let $[(u_1, \ldots, u_k)], [(v_1, \ldots, v_k)] \in \Cbf_{\mathcal{B}}$ be such that $u_1 n_1 \cdots u_k n_k B_+ = \phi([u_1, \ldots, u_k]) = \phi([v_1, \ldots, v_k]) = v_1 n_1' \cdots v_k n_k' B_+$. Let $(u_1, \ldots, u_k)$ (resp. $(v_1, \ldots, v_k)$) be of type $(s_1, \ldots, s_k)$ (resp. $(t_1, \ldots, t_k)$). Then $s_1 \cdots s_k = t_1 \cdots t_k$. Thus there exists a sequence $(w_1, \ldots, w_k) \in [v_1, \ldots, v_k]$ of type $(s_1, \ldots, s_k)$ and we have $u_1 n_1 \cdots u_k n_k B_+ = w_1 n_1 \cdots w_k n_k B_+$. The uniqueness of the decomposition in \cite[Lemma $(7.4)$]{Ro89} implies $u_i = w_i$ and hence $[u_1, \ldots, u_k] = [w_1, \ldots, w_k]$. Thus $\phi$ is a bijection. Clearly, $\phi$ preserves $s$-adjacency. Now the claim follows from \cite[Lemma $5.61$]{AB08}.
\end{proof}

We now extend the concept of a natural labelling to (arbitrary) buildings of type $A_1 \times A_1$, by defining a labelling of type $(U_1', U_2')$ to be \textit{natural} if $(u_1, u_2)$ is equivalent to $(u_2, u_1)$ for all $u_1 \in U_1, u_2 \in U_2$. If $\Delta$ is a spherical Moufang building with a natural labelling given by $1 \neq e_s \in U_{\alpha_s}$ then any $A_1 \times A_1$ residue acquires a natural labelling in this sense (because the appropriate root groups commute).

\begin{lemma}\label{reduciblerealisable}
	Let $(W, S)$ be a reducible $2$-spherical Coxeter system of rank $3$. Then a blueprint of type $(W, S)$ is realisable if the labelling of the restriction to any $A_1 \times A_1$ residue is natural.
\end{lemma}
\begin{proof}
	Let $\mathcal{B} = \left( (U_s')_{s\in S}, (\Delta_{st})_{s\neq t \in S} \right)$ be a blueprint of type $(W, S)$. Let $S = \{r, s, t\}$ and assume $m_{sr} = 2 = m_{tr}$. By Corollary \ref{Corollary: Ro89Theorem7.1+7.2} it suffices to show that for any two sequences $\bar{u},\bar{v}$ of the same reduced type, $\bar{u} \simeq \bar{v}$ implies $\bar{u} = \bar{v}$. Therefore let $\bar{u} = (u_1, \ldots, u_k)$ and $\bar{v} = (v_1, \ldots, v_k)$ be two sequences of the same reduced type $(s_1, \ldots, s_k)$ such that $\bar{u} \simeq \bar{v}$ (note that $k \leq m_{st} +1$). If $(W, S)$ is of type $A_1 \times A_1 \times A_1$ the claim follows, because an elementary equivalence is just a permutation of the sequence and for each $s\in S$ there is at most one element of $U_s$ in such a sequence. Thus we can assume that $(W, S)$ is not of type $A_1 \times A_1 \times A_1$ and hence $m_{st} \geq 3$. Since $m_{sr} = 2 = m_{tr}$ we know that $r$ occurs at most once in the reduced type. If $u_i$ is in $U_r$, then the sequence with $u_i, u_{i-1}$ (resp. $u_i, u_{i+1}$) reversed is equivalent to $(u_1, \ldots, u_k)$, since $u_i$ is the only element in $U_r$ in the sequence $\bar{u}$. We also note that we can do an elementary equivalence in $\Delta_{st}$ only if $u_i$ is at position $1$ or $k$ in the sequence. If we do an elementary equivalence in $\Delta_{st}$ we have to do this twice (because of the type) and get the same sequence as we started with. Thus the claim follows.
\end{proof}

\subsection*{The action via left multiplication}

Let $\Delta = (\C, \delta)$ be a spherical Moufang building of type $(W, S)$. For each $s\in S$ we fix $1 \neq e_s \in U_{\alpha_s}$ and put $n_s := m(e_s)$. As we have already mentioned, every chamber of $\Delta$ can be written in the form $u_1n_1 \cdots u_kn_k B_+$, where $u_i \in U_{\alpha_{s_i}}$ and this decomposition is unique if one fixes the type $(s_1, \ldots, s_k)$. Since left multiplication of $G$ is an action on $\Delta$, we want to know how $g.u_1n_1 \cdots u_kn_k B$ looks like. Assume that $\Delta$ satisfies Condition $\lco$. Since $U_+ = \langle U_{\alpha_s} \mid s\in S \rangle$ as a consequence of Lemma \ref{coUplus} and $U_{-\alpha_s}^{n_s} = U_{\alpha_s}$ it suffices to consider this action for $U_{\alpha_s}, H_s$ and $n_s^{\pm 1}$ for each $s\in S$. For $s_1, \ldots, s_k \in S$ and $u_i \in U_{\alpha_{s_i}}$ we denote the chamber $u_1n_1 \cdots u_kn_kB_+$ by $(u_1, \ldots, u_k)$. We remark that we do not consider all chambers $g.(u_1, \ldots, u_k)$ for $g\in G$.

\begin{theorem}\label{MoufangLeftmultiplicationAction}
	Let $t, s_1, \ldots, s_k \in S$ and $u_i \in U_{\alpha_{s_i}}$. Let $\omega: G \times \Delta \to \Delta$ be the left multiplication of $G$ on $\Delta$. Then we have $\omega(g, ()) = ()$ for $g\in B_+$ and $\omega(g, (u_1, \ldots, u_k))$ is given by the following:
	\allowdisplaybreaks
	\begin{align*}
	(u_1^{g^{-1}}, \omega(g^{n_{s_1}}, (u_2, \ldots, u_k))) & \text{ if } g\in H_s \text{ for some } s\in S, \\
	(gu_1, u_2, \ldots, u_k) & \text{ if } g \in U_{\alpha_{s_1}}, \\
	(u_1, \omega(g^{n_{s_1}} [g, u_1]^{n_{s_1}}, (u_2, \ldots, u_k))) & \text{ if } g \in U_{\alpha_t}, s_1 \neq t\in S, \\
	\omega(n_s, \omega(n_s^{-2}, (u_1, \ldots, u_k))) & \text{ if } g = n_s^{-1} \\
	\omega(n_{s_1}^2, (u_2, \ldots, u_k)) & \text{ if } g = n_{s_1}, u_1 = 1, \\
	(\overline{u}_1, \omega(b(u_1), (u_2, \ldots, u_k))) & \text{ if } g = n_{s_1}, u_1 \neq 1, \\
	(1_{U_t}, u_1, \ldots, u_k) & \text{ if } g = n_t, l(ts_1 \cdots s_k) = k+1.
	\end{align*}
\end{theorem}
\begin{proof}
	This is a straight forward computation.
\end{proof}

\section{Foundations and enveloping groups}\label{Section: Foundations}

This subsection is based on \cite[Ch. $11$]{WenDiss} and \cite{Mu99}.

\subsection*{Foundations}

A \textit{foundation of type $(W, S)$} is a triple $\mathcal{F} := ( (\Delta_J)_{J \in E(S)}, (c_J)_{J \in E(S)}, (\phi_{rst})_{\{r, s\}, \{s, t\} \in E(S)})$ such that the following hold:
\begin{enumerate}[label=(F\arabic*), leftmargin=*]
	\item $\Delta_J = (\C_J, \delta_J)$ is a building of type $(\langle J \rangle, J)$ with $c_J \in \C_J$ for each $J \in E(S)$.
	
	\item Each \textit{glueing} $\phi_{rst}: \P_s(c_{\{r, s\}}) \to \P_s(c_{\{s, t\}})$ is a base-point preserving bijection.
	
	\item The $\phi_{rst}$ satisfy the cocycle condition $\phi_{tsu} \circ \phi_{rst} = \phi_{rsu}$.
\end{enumerate}
It follows from the definition that $\phi_{tst} = \id$ and that $\phi_{rst} = \phi_{tsr}^{-1}$. We say that the foundation $\mathcal{F}$ satisfies Condition $\lco$ if $\Delta_J$ satisfies Condition $\lco$ for each $J \in E(S)$. For each $J\subseteq S$ the \textit{$J$-residue} of $\mathcal{F}$ is the foundation $\mathcal{F}_J := ( (\Delta_I)_{I \in E(J)}, (c_I)_{I \in E(J)}, (\phi_{rst})_{\{r, s\}, \{s, t\} \in E(J)})$ of type $(\langle J \rangle, J)$. Two foundations $\mathcal{F} = ( (\Delta_J)_{J \in E(S)}, (c_J)_{J \in E(S)}, (\phi_{rst})_{\{r, s\}, \{s, t\} \in E(S)})$ and $\mathcal{F}' = ( (\Delta_J')_{J \in E(S)}, (c_J')_{J \in E(S)}, (\phi_{rst}')_{\{r, s\}, \{s, t\} \in E(S)})$ of the same type $(W, S)$ are called \textit{isomorphic} if there exist isomorphisms $\alpha_J: \Delta_J \to \Delta_J'$ for all $J \in E(S)$ such that $\alpha_J(c_J) = c_J'$ and for all $r, s, t \in S$ with $\{r, s\}, \{s, t\} \in E(S)$ we have $\phi_{rst}' \circ \alpha_{\{r, s\}} = \alpha_{\{s, t\}} \circ \phi_{rst}$.

\begin{remark}
	We remark that there is a notion of more general isomorphisms between foundations, which allow isomorphisms of the Coxeter system. In that sense our isomorphisms are called \textit{special}.
\end{remark}

Let $\mathcal{F} = ( (\Delta_J)_{J \in E(S)}, (c_J)_{J \in E(S)}, (\phi_{rst})_{\{r, s\}, \{s, t\} \in E(S)})$ be a foundation of type $(W, S)$. An \textit{apartment} of $\mathcal{F}$ is a tupel $(\Sigma_J)_{J \in E(S)}$ such that the following hold:
\begin{enumerate}[label=(FA\arabic*), leftmargin=*]
	\item $\Sigma_J$ is an apartment of $\Delta_J$ containing $c_J$ for each $J \in E(S)$.
	
	\item Given $\{ r, s \}, \{s, t\} \in E(S)$, then $\phi_{rst}( \Sigma_{\{r, s\}} \cap \P_s(c_{\{r, s\}}) ) = \Sigma_{\{s, t\}} \cap \P_s(c_{\{s, t\}})$.
\end{enumerate}

\subsection*{Moufang sets}

A \textit{Moufang set} is a pair $\left( X, (U_x)_{x\in X} \right)$, where $X$ is a set with $\vert X \vert \geq 3$ and for each $x\in X$, $U_x$ is a subgroup of $\Sym(X)$ (we compose from right to left) such that the following hold:
\begin{enumerate}[label=(MS\arabic*), leftmargin=*]
	\item For each $x\in X$, $U_x$ fixes $x$ and acts simply transitive on $X \backslash \{x\}$.
	
	\item For all $x, y \in X$ and each $g\in U_x$, $g \circ U_y \circ g^{-1} = U_{g(y)}$.
\end{enumerate}
The groups $U_x$ for $x\in X$ are called the \textit{root groups} of the Moufang set. Let $\left( X, ( U_x )_{x\in X} \right)$ and $\left( X', ( U_{x'})_{x'\in X'} \right)$ be two Moufang sets and let $\phi:X \to X'$ be a map. Then the Moufang sets are called \textit{$\phi$-isomorphic}, if $\phi$ is bijective and for all $x\in X$ we have $\phi \circ U_x \circ \phi^{-1} = U_{\phi(x)}$.

\begin{example}
	Let $\Delta$ be a thick, irreducible, spherical Moufang building of rank at least $2$. Let $P$ be a panel, let $p\in P$ and let $\Sigma$ be an apartment in $\Delta$ with $p \in \Sigma$. Let $\alpha$ denote the unique root in $\Sigma$ containing $p$ but not $P \cap \Sigma$. Let $U_p := \{ g \vert_P  \mid g \in U_{\alpha} \}$. Then the group $U_p$ is independent of the choice of the apartment $\Sigma$ and $\mathbb{M}(\Delta, P) := \left( P, ( U_p )_{p\in P} \right)$ is a Moufang set (cf. \cite[Notation $1.19$]{MPW15}).
\end{example}

\subsection*{Moufang foundations}

A foundation $\mathcal{F} = ( (\Delta_J)_{J \in E(S)}, (c_J)_{J \in E(S)}, (\phi_{rst})_{\{r, s\}, \{s, t\} \in E(S)})$ of $2$-spherical type $(W, S)$ is called \textit{Moufang foundation}, if the following hold:
\begin{enumerate}[label=(MF\arabic*), leftmargin=*]
	\item $\Delta_J$ is a Moufang building of type $(\langle J \rangle, J)$ for each $J \in E(S)$.
	
	\item Given $\{r, s\}, \{s, t\} \in E(S)$, the Moufang sets $\mathbb{M}(\Delta_{\{r, s\}}, \P_s(c_{\{r, s\}}))$ and $\mathbb{M}(\Delta_{\{s, t\}}, \P_s(c_{\{s, t\}}))$ are $\phi_{rst}$-isomorphic.
\end{enumerate}

\begin{example}
	Let $\Delta = (\Delta_+, \Delta_-, \delta_*)$ be a thick twin building of irreducible $2$-spherical type $(W, S)$ and of rank at least $3$. Let $\epsilon \in \{+,-\}$ and $c\in \C_{\epsilon}$. By \cite[$(8.3)$ Theorem $4$]{Ro00} the residue $R_J(c)$ with the restriction of the distance function is a Moufang building for each $J \in E(S)$ and one can verify that $( (\Delta_J)_{J \in E(S)}, (c_J)_{J \in E(S)}, (\phi_{rst})_{\{r, s\}, \{s, t\} \in E(S)} )$ is a Moufang foundation of type $(W, S)$, where $\Delta_J = (R_J(c), \delta_{\epsilon}), c_J := c$ and $\phi_{rst} = \id$. We will denote this Moufang foundation by $\mathcal{F}(\Delta, c)$. It is a (non-trivial) fact that for any $\epsilon \in \{+,-\}$ and $c, c' \in \C_{\epsilon}$ we have $\mathcal{F}(\Delta, c) \cong \mathcal{F}(\Delta, c')$.
\end{example}

A foundation $\mathcal{F}$ is called \textit{integrable}, if there exists a twin building $\Delta$ of type $(W, S)$ and a chamber $c$ of $\Delta$ such that $\mathcal{F} \cong \mathcal{F}(\Delta, c)$. A foundation satisfies Condition $\lsco$ if it is $3$-spherical and if there exists a twin building $\Delta = (\Delta_+, \Delta_-, \delta_*)$ and a chamber $c$ of $\Delta$ such that $\mathcal{F} \cong \mathcal{F}(\Delta, c)$ and both buildings $\Delta_+, \Delta_-$ satisfy Condition $\lsco$. In particular, if a foundation satisfies Condition $\lsco$, it is integrable.

Let $\mathcal{F}$ be an integrable Moufang foundation. Then every panel contains at least $3$ chambers. Since $\mathcal{F}$ is integrable, there exists a twin building $\Delta$ and a chamber $c$ of $\Delta$ such that $\mathcal{F} \cong \mathcal{F}(\Delta, c)$. By Lemma \ref{E1thick} the twin building $\Delta$ is thick. Moreover, every irreducible integrable Moufang foundation satisfying Condition $\lco$ determines the isomorphism class of the corresponding twin building:

\begin{proposition}\label{Proposition: isomorphism of foundations yields iso of twin buildings}
	Let $\Delta = (\Delta_+, \Delta_-, \delta_*), \Delta' = (\Delta_+', \Delta_-', \delta_*')$ be two thick irreducible $2$-spherical twin buildings of type $(W, S)$ and of rank at least $3$ satisfying Condition $\lco$. Suppose $c\in \Delta_+, c' \in \Delta_+'$ such that $\mathcal{F}(\Delta, c) \cong \mathcal{F}(\Delta', c')$. Then $\Delta \cong \Delta'$.
\end{proposition}
\begin{proof}
	This is a consequence of \cite[Theorem $11.1.12$]{WenDiss} and \cite[Theorem $1.5$]{MR95}.
\end{proof}

\subsection*{The Steinberg group associated with an RGD-system}

Let $\mathcal{D} = (G, (U_{\alpha})_{\alpha \in \Phi})$ be an RGD-system of irreducible spherical type $(W, S)$ and rank at least $2$. Following \cite{Ti87}, the \textit{Steinberg group} associated with $\mathcal{D}$ is the group $\hat{G}$ which is the direct limit of the inductive system formed by the groups $U_{\alpha}$ and $U_{[\alpha, \beta]} := \langle U_{\gamma} \mid \gamma \in [\alpha, \beta] \rangle$ for all prenilpotent pairs $\{ \alpha, \beta \} \subseteq \Phi$. Then $U_{\alpha}, U_{[\alpha, \beta]} \to \hat{G}$ are injective and $(\hat{G}, (\hat{U}_{\alpha})_{\alpha \in \Phi})$ is an RGD-system of type $(W, S)$, where $\hat{U}_{\alpha}$ denotes the image of $U_{\alpha}$ in $\hat{G}$. Let $\Delta$ be the associated spherical building, i.e. $\Delta = \Delta(\mathcal{D})_+$, where $\Delta(\mathcal{D}) = (\Delta(\mathcal{D})_+, \Delta(\mathcal{D})_-, \delta_*)$ is the associated twin building. Then the kernel of $\hat{G} \to \Aut(\Delta)$ is equal to the center of $\hat{G}$ by \cite[Proposition $7.127(2)$]{AB08}.

\subsection*{The direct limit of a foundation}

In this subsection we let $\mathcal{F} = ((\Delta_J)_{J \in E(S)}, (c_J)_{J \in E(S)}, (\phi_{rst})_{\{r, s\}, \{s, t\} \in E(S)} )$ be a Moufang foundation of type $(W, S)$ satisfying Condition $\lco$ and let $(\Sigma_J)_{J \in E(S)}$ be an apartment of $\mathcal{F}$. As $\mathcal{F}$ is a Moufang foundation, the buildings $\Delta_J$ are Moufang buildings. We identify the roots in $\Sigma_J$ with $\Phi^J$. For $\alpha \in \Phi^J$ we let $U_{\alpha}^J$ be the root group associated with $\alpha \subseteq \Sigma_J$ and we let $H_J = \langle U_{\alpha}^J \mid \alpha \in \Phi^J \rangle \leq \Aut(\Delta_J)$. As $J \in E(S)$, it follows from \cite[Remark $7.107(a)$]{AB08} that the root groups $U_{\alpha}^J$ are nilpotent.

We note that for each $\{s, t\} \in E(S)$ the restriction $U_{\pm \alpha_s}^{\{s, t\}} \to U_{\pm \alpha_s}^{\{s, t\}} \vert_{\P_s(c_{\{s, t\}})}$ is an isomorphism. For each $s\in S$ we fix $s_0 \in S$ such that $m_{ss_0} > 2$ and we define $U_{\pm s} := U_{\pm \alpha_s}^{\{s, s_0\}} \vert_{\P_s(c_{\{s, s_0\}})}$. Using (MF$2$) we know that for each $t\in S$ with $\{s, t\} \in E(S)$ the mapping $U_{\pm s} \to U_{\pm \alpha_s}^{\{s, t\}} \vert_{\P_s(c_{\{s, t\}})}, g \mapsto \phi_{s_0 st} \circ g \circ \phi_{s_0 st}^{-1}$ is an isomorphism. Thus we have canonical isomorphisms $U_{\pm s} \to U_{\pm \alpha_s}^{\{s, t\}}$.

Let $J = \{s, s_0\} \in E(S)$. For each $s\in K \in E(S)$ we denote the image of $u \in U_{\pm s}$ in $\Aut(\Delta_K)$ by $u_K$. Then for every $1 \neq u \in U_s$ there exist $u', u'' \in U_{-s}$ such that $(u')_J u_J (u'')_J$ stabilises $\Sigma_J$ and acts on $\Sigma_J$ as the reflection $r_{\alpha_s}$. By \cite[Consequence $(3)$ on p.$415$]{AB08} $u', u''$ are unique. Let $s\in K \in E(S)$. By construction of $U_{\pm s} \to \Aut(\Delta_K)$ we know that $(u')_K u_K (u'')_K$ interchanges the elements in $\Sigma_K \cap \P_s(c_K)$ and stabilizes every panel $P$ with $\vert P \cap \Sigma_K \vert = 2$. As $K$ is spherical, this implies that $(u')_K u_K (u'')_K$ stabilizes $\Sigma_K$ and acts on $\Sigma_K$ as the reflection $r_{\alpha_s}$ (cf. \cite[proof of Lemma $7.5$]{AB08}). This implies that for $u \in U_s$ the elements $u', u'' \in U_{-s}$ do not depend on the choice of $s\in K \in E(S)$. We define for every $1 \neq u \in U_s$ the element $m(u) := u' u u''$, where $u', u'' \in U_{-s}$ are as above.

\begin{lemma}\label{Lemma: X_s RGD-system}
	Let $\pi_s: U_s \star U_{-s} \to \prod_{s\in J \in E(S)} \Aut(\Delta_J)$ be the canonical homomorphism and let $K_s := \ker \pi_s$. Let $X_s := \left( U_s \star U_{-s} \right)/K_s$. Then $U_{\pm s} \to X_s$ is injective and $(X_s, (U_{\pm s}))$ is an RGD-system of type $A_1$, where we identify $U_{\pm s}$ with its image in $X_s$.
\end{lemma}
\begin{proof}
	Note that $U_{\pm s} \to \Aut(\Delta_J)$ are injective for every $s\in J \in E(S)$ and $U_{-\alpha_s}^J \not\leq U_{\alpha_s}^J$. Thus $U_{\pm s} \to X_s$ are injective and $U_{-s} \not\leq U_s$ in $X_s$. It suffices to show (RGD$2$). We show that $m(u) = u' u u''$ conjugates $U_{\pm s}$ to $U_{\mp s}$ for every $1 \neq u \in U_s$. Let $J = \{s, s_0\}$ and $x \in U_s$. Then $m(u)_J^{-1} x_J m(u)_J \in U_{-\alpha_s}^J$ and hence $m(u)_J^{-1} x_J m(u)_J = (x')_J$ for some $x' \in U_{-s}$. By construction of $U_{\pm s} \to \Aut(\Delta_K)$ we infer $m(u)_K^{-1} x_K m(u)_K = (x')_K$. Thus $(x')^{-1} m(u)^{-1} x m(u) \in K_s$ for every $x\in U_s$ and hence $m(u)^{-1} x m(u) K_s = x' K_s$ in $X_s$. This implies that $(X_s, (U_{\pm s}))$ is an RGD-system of type $A_1$.
\end{proof}

\begin{example}\label{Example:Xs injective}
	Let $\Delta$ be an irreducible twin building satisfying Condition $\lco$. Using \cite[Theorem $1.5$]{MR95} and \cite[Theorem $8.27$]{AB08}, $\Delta$ is a so-called \textit{Moufang twin building}. Let $c_+ \in \Delta_+, c_- \in \Delta_-$ and let $\Sigma$ be the twin apartment containing $c_+$ and $c_-$. Let $U_{\alpha}$ be the corresponding root groups and let $G = \langle U_{\alpha} \mid \alpha \in \Phi \rangle \leq \Aut(\Delta)$. Then $(G, (U_{\alpha})_{\alpha \in \Phi})$ is an RGD-system by \cite[$8.47(a)$]{AB08}. 
	
	Let $\mathcal{F} = \mathcal{F}(\Delta, c_+)$ and $\Sigma_J := R_J(c_+) \cap \Sigma$. Then $(\Sigma_J)_{J \in E(S)}$ is an apartment of $\mathcal{F}$. We consider the homomorphism $\phi: U_s \star U_{-s} \to \Aut(\Delta)$ be the canonical homomorphism. Let $g\in \ker \phi$. Then $g$ acts trivial on every rank $2$ residue and hence $g \in \ker \pi_s$. Now let $g\in \ker \pi_s$. Then $\phi(g) \in \langle U_{\alpha_s} \cup U_{-\alpha_s} \rangle$. As $\phi(g)$ fixes $\P_s(c_+)$, we deduce $\phi(g) \in H_{\alpha_s}$ and hence $\phi(g)$ fixes $\Sigma$. If $m_{st} >2$, then $\phi(g)$ fixes $R_{\{s, t\}}(c_+)$ by assumption. If $m_{st} = 2$, then $\phi(g)$ fixes $\P_t(c_+)$, as the corresponding root groups commute. In particular, $\phi(g)$ fixes a twin apartment and all neighbours of $c_+$. Thus $\phi(g) = 1$ by \cite[Theorem $1$]{Ro00} and we have $\ker \pi_s = \ker \phi$. In particular, $X_s \to \Aut(\Delta)$ is injective, where $X_s$ is as in the previous lemma.
\end{example}

Let $J = \{s, t\} \in E(S)$ and let $\hat{H}_J$ be the Steinberg group associated with $(H_J, (U_{\alpha}^J)_{\alpha \in \Phi^J})$. Let $\pi_J^{\mathrm{St}}: \hat{H}_J \to \Aut(\Delta_J), \pi_{s,J}: U_s \star U_{-s} \to \Aut(\Delta_J), \phi_s: U_s \star U_{-s} \to \hat{H}_J$ be the canonical homomorphisms. As $\pi_s: U_s \star U_{-s} \to \prod_{s\in K \in E(S)} \Aut(\Delta_K)$, we deduce $\ker \pi_s \leq \ker \pi_{s,J}$. As $\pi_{s,J} = \pi_J^{\mathrm{St}} \circ \phi_s$, we have $\phi_s(\ker \pi_s) \leq \phi_s( \ker \pi_{s,J} ) \leq \ker \pi_J^{\mathrm{St}}$. Since $\ker \pi_J^{\mathrm{St}} = Z(\hat{H}_J)$, we deduce $K_{st} := \phi_s(\ker \pi_s) \phi_t( \ker \pi_t ) \trianglelefteq \hat{H}_J$. Now $\hat{H}_J/K_{st}$ is again an RGD-system by \cite[$7.131$]{AB08}. Thus we let $\mathcal{D}_J = (\hat{H}_J/K_{st}, (\hat{U}_{\alpha}K_{st}/K_{st})_{\alpha \in \Phi^J})$ and write $X_{s, t} := \hat{H}_J/K_{st}$. Let $\psi_{st}: \hat{H}_J \to X_{s, t}$ be the canonical homomorphism. Then $\left( \psi_{st} \circ \phi_s\right)(\ker \pi_s) \leq \psi_{st}(K_{st}) = 1$ and hence $\psi_{st} \circ \psi_s$ factors through $X_s \to \hat{H}_J/K_{st}$.
\begin{center}
	\begin{tikzcd}
		&X_s = (U_s \star U_{-s}) / \ker \pi_s \arrow[rd] & \\
		U_s \star U_{-s} \arrow[ru] \arrow[r, "\phi_s"] \arrow[rd, "\pi_{s,J}"] & \hat{H}_J \arrow[r, "\psi_{st}"] \arrow[d, "\pi_J^{\mathrm{St}}"] & \hat{H}_J/K_{st} = X_{s, t} \arrow[ld] \\
		&\Aut(\Delta_J)&
	\end{tikzcd}
\end{center}

If $s\neq t\in S$ are such that $m_{st} = 2$, we define $X_{s, t} := X_s \times X_t$. We define $G$ to be the direct limit of the inductive system formed by $X_s, X_{s, t}$ for all $s\neq t\in S$. Note that $X_{s, t}$ is generated by the canonical image of $X_s, X_t$ in $X_{s, t}$ and hence $G = \langle X_s \mid s\in S \rangle$. We note that it is not clear whether $X_s \to G$ is injective. If we write $1 \neq u \in U_s$ we simply mean that $u \neq 1$ in the group $U_s$.

\begin{lemma}
	Let $s_1, \ldots, s_k, s \in S$ and let $u_i, v_i \in U_{s_i} \backslash \{1\}$. Then $U_s^{m(u_1) \cdots m(u_k)} = U_s^{m(v_1) \cdots m(v_k)}$. 
\end{lemma}
\begin{proof}
	We show the hypothesis by induction on $k$. For $k=1$ we have $m(v_1)^{-1} m(u_1) \in H_{s_1} \leq N_{X_{s_1, s}}(U_s)$. Thus we assume $k>1$. Using \cite[Consequences $(4)$ and $(5)$ on p.$415$]{AB08} we deduce $tm(u_i)t^{-1} = m(tu_it^{-1})$ and $tm(u_i)^{-1} t^{-1} = m(tu_it^{-1})^{-1}$ for each $t\in H_t$ and $u_i \in U_{s_i}$. Note that $tu_i t^{-1} \in U_{s_i}$. Using induction, this implies for $t := m(v_k) m(u_k)^{-1} \in H_{s_k}$:
	\allowdisplaybreaks
	\begin{align*}
		&t m(u_{k-1})^{-1} \cdots m(u_1)^{-1} U_s m(u_1) \cdots m(u_{k-1}) t^{-1} \\
		&\qquad = m(tu_{k-1}t^{-1})^{-1} \cdots m(tu_1t^{-1})^{-1} t U_s t^{-1} m(tu_1t^{-1}) \cdots m(tu_{k-1}t^{-1}) \\
		&\qquad = m(tu_{k-1}t^{-1})^{-1} \cdots m(tu_1t^{-1})^{-1} U_s m(tu_1t^{-1}) \cdots m(tu_{k-1}t^{-1}) \\
		&\qquad = m(v_{k-1})^{-1} \cdots m(v_1)^{-1} U_s m(v_1) \cdots m(v_{k-1}) \qedhere
	\end{align*}
\end{proof}

\begin{lemma}
	Suppose $s, t, s_1, \ldots, s_k, t_1, \ldots, t_l \in S$ such that $s_1 \cdots s_k \alpha_s = t_1 \cdots t_l \alpha_t$. Let $1 \neq u_i \in U_{s_i}, 1 \neq v_i \in U_{t_i}$. Then $U_s^{m(u_k) \cdots m(u_1)} = U_t^{m(v_l) \cdots m(v_1)}$.
\end{lemma}
\begin{proof}
	At first we show that we have an action of $W$ on the conjugacy class of $U_s$, where $s\in S$ acts on every conjugacy class by conjugation with $m(u)$ for some $1 \neq u \in U_s$. By the previous lemma, conjugation with $m(u)$ does not depend on $1 \neq u \in U_s$ and $m(u)^2$ acts trivial on every conjugacy class. Moreover, $(st)^{m_{st}}$ acts as $\left( m(u_s) m(u_t) \right)^{m_{st}} \in \langle H_s \cup H_t \rangle$ (cf. \cite[Lemma $3.3$]{Ca06}) and hence trivial.
	
	We are now in the position to prove the claim. If $s\neq t$, there exists $w\in \langle s, t \rangle$ such that $w\alpha_s = \alpha_t$ and $w^{-1} U_s w = U_t$ holds in $X_{s, t}$. Thus we can assume $s=t$. It suffices to show that for $w\in W$ with $w\alpha_s = \alpha_s$ we have $w^{-1} U_s w = U_s$. Thus let $w\in W$ such that $w\alpha_s = \alpha_s$. Then $ws \notin \alpha_s$ and hence $\{ w, ws \} \in \partial \alpha_s$. By Lemma \ref{CM06Prop2.7} there exist a sequence of panels $P_0 = \{ 1, s \}, \ldots, P_n = \{ w, ws \}$ contained in $\partial \alpha_s$ and a sequence of spherical rank $2$ residues $R_1, \ldots, R_n$ contained in $\partial^2 \alpha_s$ such that $P_{i-1}, P_i$ are distinct and contained in $R_i$. We show that claim via induction on $n$. For $n = 0$ there is nothing to show. Let $n>0$ and let $x \in P_{n-1} \cap \alpha_s$. Then $x^{-1}w \in \langle J \rangle$ for some $J \subseteq S$ with $\vert J \vert = 2$. Using induction, we deduce $w^{-1} U_s w = (x^{-1}w)^{-1} x^{-1} U_s x (x^{-1}w) = (x^{-1}w)^{-1} U_s (x^{-1}w) = U_s$. This finishes the claim.
\end{proof}

For $s\in S$ we define $U_{\alpha_s} := U_s$ and for $\alpha \in \Phi$ we define $U_{\alpha}$ as a conjugate of $U_s$ as in (RGD$2$). This is well-defined by the previous lemma. We put $\mathcal{D}_{\mathcal{F}} := (G, (U_{\alpha})_{\alpha \in \Phi})$.

\begin{lemma}\label{RGD24-1fornonnested}
	The system $\mathcal{D}_{\mathcal{F}}$ satisfies (RGD$2$) and (RGD$4$). Furthermore, (RGD$1$) holds for each pair $\{ \alpha, \beta \}$ of prenilpotent roots with $o(r_{\alpha} r_{\beta}) < \infty$.
\end{lemma}
\begin{proof}
	It follows by the definition of $G$ and $U_{\alpha}$ that (RGD$2$) and (RGD$4$) are satisfied. Let $\{ \alpha, \beta \}$ be a prenilpotent pair with $o(r_{\alpha} r_{\beta}) < \infty$. Then there exists $R \in \partial^2 \alpha \cap \partial^2 \beta$. Let $w = \proj_R 1_W$. Then $w^{-1} \alpha, w^{-1} \beta \in \Phi^{\{s, t\}}$ for some $s\neq t \in S$. As $[U_{w^{-1}\alpha}, U_{w^{-1}\beta}] \leq U_{(w^{-1}\alpha, w^{-1}\beta)}$ holds in $X_{s, t}$, it also holds in $G$. This implies $[U_{\alpha}, U_{\beta}]^w = [ U_{w^{-1}\alpha}, U_{w^{-1}\beta} ] \leq U_{(w^{-1}\alpha, w^{-1}\beta)}$ and hence $[U_{\alpha}, U_{\beta}] \leq U_{(w^{-1}\alpha, w^{-1}\beta)}^{w^{-1}} = U_{(ww^{-1}\alpha, ww^{-1}\beta)} = U_{(\alpha, \beta)}$.
\end{proof}

\begin{lemma}\label{Uwnilpotentgroup}
	Let $(c_0, \ldots, c_k)$ be a minimal gallery and let $(\alpha_1, \ldots, \alpha_k)$ be the sequence of roots which are crossed by that minimal gallery. Then we have $[U_{\alpha_1}, U_{\alpha_k}] \leq \langle U_{\alpha_2}, \ldots, U_{\alpha_k} \rangle$. In particular, $U_{\alpha_1} \cdots U_{\alpha_k}$ is a nilpotent group.
\end{lemma}
\begin{proof}
	We can assume $\alpha_1 \subseteq \alpha_k$. In particular, $k \geq 3$ and $c_0, \ldots, c_k$ are not contained in a rank $2$ residue. Let $R$ be the residue containing $c_{k-2}, c_{k-1}$ and $c_k$. Then there exists a minimal gallery $(d_0 = c_0, \ldots, d_k = c_k)$ with $d_i = \proj_R c_0$ for some $0 \leq i \leq k$. We note that the set of roots crossed by $(d_0, \ldots, d_k)$ coincides with the set of roots crossed by $(c_0, \ldots, c_k)$. Let $(\beta_1, \ldots, \beta_k)$ be the set of roots crossed by $(d_0, \ldots, d_k)$ and let $1 \leq j \leq k$ be such that $\beta_j = \alpha_1$. Since $\alpha_1 \subseteq \alpha_k$, we have $j < i$. Let $\gamma \neq \beta_i$ be the root containing $d_i$ but not some neighbour of $d_i$ in $R$. Since $\{ \beta_{j+1}, \ldots, \beta_{i-1} \} \subseteq \{ \alpha_2, \ldots, \alpha_{k-1} \}$ it suffices to show $[U_{\alpha_1}, U_{\alpha_k}] \leq \langle U_{\beta_{j+1}}, \ldots, U_{\beta_{i-1}} \rangle$.	
	
	If $\alpha_k \in \{ \gamma, \beta_i \}$, then we have $[U_{\alpha_1}, U_{\alpha_k}] \leq \langle U_{\beta_{j+1}}, \ldots, U_{\beta_{i-1}} \rangle$ by induction. Thus we assume $\alpha_k \notin \{ \gamma, \beta_i \}$. By Lemma \ref{coUplus} and (RGD$2$) we deduce $U_{\alpha_k} \leq [U_{\gamma}, U_{\beta_i}]$. Then for each $u_k \in U_{\alpha_k}$ there exist $a_1, \ldots, a_n \in U_{\gamma}, b_1, \ldots, b_n \in U_{\beta_i}$ such that $u_k = a_1 b_1 \cdots a_n b_n$. Using induction, we know that $[U_{\alpha_1}, U_{\gamma}], [U_{\alpha_1}, U_{\beta_i}] \leq \langle U_{\beta_{j+1}}, \ldots, U_{\beta_{i-1}} \rangle$ as well as $[U_{\beta_l}, U_{\gamma}], [U_{\beta_l}, U_{\beta_i}] \leq \langle U_{\beta_{l+1}}, \ldots, U_{\beta_{i-1}} \rangle$ for every $j+1 \leq l \leq i-1$. Let $u_1 \in U_{\alpha_1}$. Using induction, we deduce
	\begin{align*}
		[u_1, u_k] = [u_1, a_1b_1 \cdots a_nb_n] = [u_1,b_n]  [u_1, a_1b_1 \cdots a_n]^{b_n} \in \langle U_{\beta_{j+1}}, \ldots, U_{\beta_{i-1}} \rangle
	\end{align*}
	Let $V := U_{\alpha_1} \cdots U_{\alpha_{k-1}}$. Using induction, $V \leq G$. Since $[U_{\alpha_i}, U_{\alpha_k}] \leq \langle U_{\alpha_{i+1}}, \ldots, U_{\alpha_{k-1}} \rangle \leq V$ for $1 \leq i \leq k-1$, we have $V U_{\alpha_k} = U_{\alpha_k} V$. It follows that $V U_{\alpha_k} = U_{\alpha_1} \cdots U_{\alpha_k}$ is a subgroup of $G$. In particular, the subgroup is nilpotent because of the commutator relations and the fact that the groups $U_{\delta}$ are nilpotent.
\end{proof}

\begin{lemma}\label{Embeddingofrootgroups}
	Let $(c_0, \ldots, c_k)$ be a minimal gallery and let $(\alpha_1, \ldots, \alpha_k)$ be the sequence of roots which are crossed by that minimal gallery. Then the product map $U_{\alpha_1} \times \cdots \times U_{\alpha_k} \to \langle U_{\alpha_i} \mid 1 \leq i \leq k \rangle$ is a bijection, if $\mathcal{D}_{\mathcal{F}}$ satisfies (RGD$3$).
\end{lemma}
\begin{proof}
	As in the definition of an RGD-system we let $U_+ = \langle U_{\alpha} \mid \alpha \in \Phi_+ \rangle$.	At first we assume that $U_{-\alpha_s} \cap U_+ \neq \{1\}$. Then \cite[Consequence $(1)$ on p.$414$]{AB08} would imply $U_{-\alpha_s} \leq U_+$, which is a contradiction to (RGD$3$). Thus $U_{-\alpha_s} \cap U_+ = \{1\}$.
	
	By the previous lemma we have $\langle U_{\alpha_1}, \ldots, U_{\alpha_k} \rangle = U_{\alpha_1} \cdots U_{\alpha_k}$ and hence the product map is surjective. We prove by induction on $k$ that it is injective. If $k=1$ there is nothing to show. Thus let $k \geq 2$ and let $u_i, v_i \in U_{\alpha_i}$ such that $u_1 \cdots u_k = v_1 \cdots v_k$. Then we have $ v_1^{-1} u_1 = v_2 \cdots v_k (u_2 \cdots u_k)^{-1}$. Using (RGD$2$) we can arrange it that $\alpha_1 = -\alpha_s \in \Phi_-$ for some $s\in S$ and $\alpha_i \in \Phi_+$ for every $2 \leq i \leq k$. Since $U_{-\alpha_s} \cap U_+ = 1$ we obtain $u_1 = v_1$ and $u_2 \cdots u_k = v_2 \cdots u_k$. Using induction the claim follows.
\end{proof}

\begin{theorem}\label{Theorem: Moufang foundation RGD0+3 implies RGD-system}
	If $\mathcal{D}_{\mathcal{F}}$ satisfies (RGD$0$) and (RGD$3$), then $\mathcal{D}_{\mathcal{F}}$ is an RGD-system.
\end{theorem}
\begin{proof}
	By assumption and Lemma \ref{RGD24-1fornonnested} it suffices to show (RGD$2$) for all roots $\alpha \subsetneq \beta \in \Phi$. Let $\alpha \subsetneq \beta$, let $G = (c_0, \ldots, c_k)$ be a minimal gallery and let $(\alpha_1 = \alpha, \ldots, \alpha_k = \beta)$ be the sequence of roots which are crossed by $G$. Then $P := \{ c_{k-1}, c_k \} \in \partial \beta$. Using Lemma \ref{Embeddingofrootgroups} we obtain a unique set $I \subseteq \{ 2, \ldots, k-1 \}$ and unique elements $1 \neq u_{\alpha_i} \in U_{\alpha_i}$ such that $[u_{\alpha}, u_{\beta}] = \prod_{i\in I} u_{\alpha_i}$. Let $Q \in \partial \beta$. Using Lemma \ref{CM06Prop2.7} there exist a sequence $P_0 = P, \ldots, P_n = Q$ of panels in $\partial \beta$ and a sequence $R_1, \ldots, R_n$ of spherical rank $2$ residues in $\partial^2 \beta$ such that $P_{i-1}, P_i$ are distinct and contained in $R_i$. Let $(d_0 = c_0, \ldots, d_m)$ be a minimal gallery such that $Q = \{ d_{m-1}, d_m \}$. We will show by induction on $n$ that for every $i\in I$ the root $\alpha_i$ is crossed by the gallery $(d_0, \ldots, d_m)$. For $n=0$ there is nothing to show, as two minimal galleries from one chamber to another chamber cross the same roots. Thus we assume $n>0$. Note that by \cite[Lemma $3.69$]{AB08} every minimal gallery from $c_0$ to a chamber which is not contained in $\beta$ has to cross $\partial \alpha$ and $\partial \beta$. Let $(e_0 = c_0, \ldots, e_l)$ be a minimal gallery such that $P_{n-1} = \{ e_{l-1}, e_l \}$ and $e_z = \proj_{R_n} e_0$ for some $z$. Moreover, we let $(\beta_1, \ldots, \beta_l = \beta)$ be the sequence of roots which are crossed by $(e_0, \ldots, e_l)$. Using induction $\alpha_i$ is crossed by $(e_0, \ldots, e_l)$ for every $i\in I$.
	
	Using Lemma \ref{Embeddingofrootgroups} we obtain a unique set $J \subseteq \{ 2, \ldots, l-1 \}$ and unique elements $1 \neq v_{\beta_j} \in U_{\beta_j}$ such that $[u_{\alpha}, u_{\beta}] = \prod_{j\in J} v_{\beta_j}$. Note that $\alpha \neq \beta_1$ in general. Similarly, if $(\delta_1, \ldots, \delta_m = \beta)$ is the sequence of roots crossed by $(e_0, \ldots, e_z = \proj_{R_n} e_0, \ldots, d_{m-1}, d_m)$, there exist a unique set $F \subseteq \{ 2, \ldots, m-1 \}$ and unique elements $1 \neq x_{\delta_f} \in U_{\delta_f}$ such that $[u_{\alpha}, u_{\beta}] = \prod_{f\in F} x_{\delta_f}$. Extending both galleries to $p$, where $e_z, p$ are opposite in $R_n$, Lemma \ref{Embeddingofrootgroups} implies that $J, F \subseteq \{ 2,\ldots, z-1 \}$, as $\langle U_{\beta_x} \mid z < x < l \rangle \cap \langle U_{\delta_x} \mid z < x < m \rangle = \{1\}$.
	
	We have $\prod_{i\in I} u_{\alpha_i} = \prod_{j\in J} v_{\beta_j}$. Assume $\beta_{l-1} \in \{ \alpha_i \mid i\in I \}$. Then $u_{\beta_{l-1}} \in \langle U_{\beta_j} \mid 2 \leq j \leq l-2 \rangle$ and the previous lemma yields a contradiction. Thus $\beta_{l-1} \notin \{ \alpha_i \mid i\in I \}$ and $\alpha_i$ is crossed by the gallery $(e_0, \ldots, e_{l-1})$ for every $i\in I$. Repeating that argument we deduce that for $i\in I$ the root $\alpha_i$ is crossed by $(e_0, \ldots, e_z)$ and hence by $(e_0, \ldots, e_z, \ldots, d_m)$. Using induction again, we deduce that $\alpha_i$ is crossed by $(d_0, \ldots, d_m)$ for every $i\in I$ and the claim follows.
	
	Now assume that for $i\in I$ we have $o(r_{\beta} r_{\alpha_i}) < \infty$. Then there exists $R \in \partial^2 \beta \cap \partial^2 \alpha_i$. Suppose $Q \in \partial \beta$ such that $Q \subseteq \alpha_i \cap R$ and let $\proj_Q c_0 \neq q \in Q$. Then a minimal gallery from $c_0$ to $q$ does not cross the root $\alpha_i$. This yields a contradiction and we infer $\alpha_i \subsetneq \beta$ for every $i\in I$. Using similar arguments, we deduce $\alpha \subsetneq \alpha_i$ for every $i\in I$. Since $\gamma \in (\alpha, \beta)$ if and only if $\alpha \subsetneq \gamma \subsetneq \beta$, the claim follows.
\end{proof}

\begin{corollary}\label{inclusionsinjectivergd3}
	If the canonical mappings $U_{\pm s} \to G$ are injective and if $\mathcal{D}_{\mathcal{F}}$ satisfies (RGD$3$), then $\mathcal{D}_{\mathcal{F}}$ is an RGD-system and we have $\mathcal{F} \cong \mathcal{F}(\Delta(\mathcal{D}_{\mathcal{F}}), B_+)$.
\end{corollary}
\begin{proof}
	By Theorem \ref{Theorem: Moufang foundation RGD0+3 implies RGD-system} $\mathcal{D}_{\mathcal{F}}$ is an RGD-system. Thus it suffices to show that $\mathcal{F} \cong \mathcal{F}(\Delta(\mathcal{D}_{\mathcal{F}}), B_+)$. Since $U_{\pm s} \to G$ are injective, we do not distinguish between them and their images in $G$.
	
	Each $\Delta_{\{s, t\}}$ is a spherical Moufang building for $s \neq t$ with $\{s, t\} \in E(S)$. Thus there exists an isomorphism $\beta_{\{s, t\}}: X_{s, t} /B_{\{s, t\}} \to \Delta_{\{s, t\}}, gB_{\{s, t\}} \mapsto g(c_{\{s, t\}})$ by \cite[Lemma $7.28$]{AB08}, where $B_{\{s, t\}} := \langle H_s \cup H_t \cup U_{\alpha} \mid \alpha \in \Phi_+^{\{s, t\}} \rangle$.
	
	By Remark \ref{Remark: Homomorphism between RGD-systems} and Lemma \ref{Lemma: RGD-system residue isomorphism}, we know that $\alpha_{\{s, t\}}: X_{s, t}/B_{\{s, t\}} \to R_{\{s, t\}}(B_+), gB_{\{s, t\}} \to gB_+$ is an isomorphism for every $\{s, t\} \in E(S)$. Thus we have an isomorphism $\gamma_{\{s, t\}} := \alpha_{\{s, t\}} \circ \beta_{\{s, t\}}^{-1}: \Delta_{\{s, t\}} \to R_{\{s, t\}}(B_+), d=g(c_{\{s, t\}}) \mapsto gB_+$.
	
	It remains to show that $\gamma_{\{r, s\}}(d) = \left( \gamma_{\{s, t\}} \circ \phi_{rst} \right)(d)$ for every $d\in \P_s(c_{\{r,s\}})$. Let $d = g(c_{\{r, s\}})$ for some $g\in \langle U_{\alpha_s}^{\{r, s\}} \cup U_{-\alpha_s}^{\{r, s\}} \rangle$. Let $g_i \in U_{\pm \alpha_s}^{\{r, s\}}$ such that $g = g_1 \cdots g_n$ and let $g_i' \in U_{\pm \alpha_s}^{\{s, t\}}$ be the unique element such that $g_i' \vert_{\P_s(c_{\{s, t\}})} = \phi_{rst} \circ g_i \vert_{\P_s(c_{\{r, s\}})} \circ \phi_{rst}^{-1}$. Let $g' := g_1 \dots g_n'$. Then
	\begin{align*}
		\phi_{rst}(d) = \phi_{rst}( g(c_{\{r, s\}}) ) = \phi_{rst} \left( g \vert_{\P_s(c_{\{r, s\}})} \left( \phi_{rst}^{-1} (c_{\{s, t\}}) \right) \right) = g_1' \cdots g_n' \vert_{\P_s(c_{\{s, t\}})} (c_{\{s, t\}}) = g'(c_{\{s, t\}})
	\end{align*}
	Since $g_i' = g_i$ in $G$, we deduce $g' = g$ in $G$. As $\gamma_{\{r, s\}}(d) = \gamma_{\{r, s\}}(g(c_{\{r, s\}})) = gB_+$, we have
	\allowdisplaybreaks
	\begin{align*}
		\left( \gamma_{\{s, t\}} \circ \phi_{rst} \right)(d) &= \gamma_{\{s, t\}} \left( \phi_{rst}(d) \right) = \gamma_{\{s, t\}} \left( g' (c_{\{s, t\}}) \right) = g' B_+ = gB_+ = \gamma_{\{r, s\}}(d) \qedhere
	\end{align*}
\end{proof}

For $J \subseteq S$ we define $G_J$ to be the direct limit of the groups $X_s$ and $X_{s, t}$ with the inclusions as above, where $s \neq t \in J$. It follows directly from the definition of the direct limit that we have a homomorphism $G_J \to G$ extending $X_s, X_{s, t} \to G$. We define $(\mathcal{D}_{\mathcal{F}})_J = (G_J, (U_{\alpha})_{\alpha \in \Phi^J})$ and note that we have $\mathcal{D}_{(\mathcal{F}_J)} \neq (\mathcal{D}_{\mathcal{F}})_J$ in general, as the group $X_s$ depends on $\mathcal{F}$.

\begin{lemma}\label{Lemma: integrable implies RGD-system}
	Let $\mathcal{F}$ be an irreducible Moufang foundation of type $(W, S)$ satisfying Condition $\lco$. If $J \subseteq S$ is irreducible such that $\vert J \vert \geq 3$ and $\mathcal{F}_J$ is integrable, then the following hold:
	\begin{enumerate}[label=(\alph*)]
		\item Let $\Delta$ be a twin building of type $(W, S)$ and let $c$ be a chamber of $\Delta$ such that $\mathcal{F} \cong \mathcal{F}(\Delta, c)$. Then we have a canonical homomorphism $G_J \to \Aut(\Delta)$
		
		\item The homomorphisms $U_{\pm s} \to G_J$ are injective for every $s\in J$ and $(\mathcal{D}_{\mathcal{F}})_J$ is an RGD-system.
	\end{enumerate}
\end{lemma}
\begin{proof}
	Since $\mathcal{F}_J$ is integrable, there exists a thick twin building $\Delta$ of type $(W, S)$ and a chamber $c$ of $\Delta$ such that $\mathcal{F}_J \cong \mathcal{F}(\Delta, c)$. Using \cite[Theorem $1.5$]{MR95} and \cite[Theorem $8.27$]{AB08}, $\Delta$ is a so-called \textit{Moufang twin building}. Let $\epsilon \in \{+, -\}$ be such that $c\in \C_{\epsilon}$. By Lemma \ref{Lemma: Apartment of foundation is contained in Apartment of twin building} there exists a twin apartment $\Sigma$ containing the image of the apartments $\Sigma_J$ of the foundation. As we have seen in Example \ref{Example:Xs injective}, $K_s$ acts trivial $\Delta$ and we have homomorphism $X_s \to \Aut(\Delta)$. Clearly, $U_{\pm s} \to \Aut(\Delta)$ are injective. Let $\{ \alpha, \beta \}  \subseteq \Phi^{\{s, t\}}$ be a prenilpotent pair. Note that restriction of an automorphism of $\Delta$ to $R_{\{s, t\}}(c)$ is an epimorphism. Using \cite[Corollary $7.66$]{AB08} it is an isomorphism from $U_{[\alpha, \beta]}$ to its image in $\Aut(R_{\{s, t\}}(c))$. Thus we have a homomorphism $\hat{H}_{\{s, t\}} \to \Aut(\Delta)$ for $m_{st} >2$. As $K_s$ is contained in the kernel of this map, we obtain a homomorphism $X_{s, t} = \hat{H}_{\{s, t\}}/K_{st} \to \Aut(\Delta)$ for $m_{st}>2$. As $U_{\pm \alpha_s}$ commutes with $U_{\pm \alpha_t}$ in $\Aut(\Delta)$ if $m_{st} = 2$, we also have a homomorphism $X_{s, t} = X_s \times X_t \to \Aut(\Delta)$. We conclude that there exists a homomorphism $G_J \to \Aut(\Delta)$ mapping $U_{\pm s}$ onto $U_{\pm \alpha_s}$. Note that $U_{\alpha} \leq \langle U_{\alpha_s} \mid s\in S \rangle$ fixes $c$ for every $\alpha \in \Phi_+^J$, whereas $U_{-\alpha_s}$ does not. Thus (RGD$3$) holds and the claim follows from Theorem \ref{Theorem: Moufang foundation RGD0+3 implies RGD-system}.
\end{proof}

\begin{lemma}\label{Lemma: reducible implies RGD-system}
	Let $\mathcal{F}$ be a Moufang foundation and let $J \subseteq S$ be reducible such that $\vert J \vert = 3$. Then $U_{\pm s} \to G_J$ are injective and $(\mathcal{D}_{\mathcal{F}})_J$ is an RGD-system.
\end{lemma}
\begin{proof}
	Let $J = \{r, s, t\}$ and assume $m_{rs} = 2 = m_{rt}$.	If $(\langle J \rangle, J)$ is of type $A_1 \times A_1 \times A_1$, then $G_{\{r, s, t\}} = X_r \times X_s \times X_t$ is an RGD-system. Otherwise, $G_{\{r, s, t\}} = X_r \times X_{s, t}$ is an RGD-system.
\end{proof}

\begin{lemma}
	Let $\mathcal{F}$ be an irreducible Moufang foundation of type $(W, S)$. Then $\mathcal{F}$ is integrable if and only if $\mathcal{D}_{\mathcal{F}}$ is an RGD-system and $U_{\pm s} \to G$ are injective.
\end{lemma}
\begin{proof}
	One implication is Corollary \ref{inclusionsinjectivergd3}; the other follows from Lemma \ref{Lemma: integrable implies RGD-system} applied to $J=S$.
\end{proof}

\begin{lemma}
	Assume that for every irreducible $J \subseteq S$ with $\vert J \vert = 3$ the $J$-residue $\mathcal{F}_J$ is integrable. Then $X_s \to X_{s, t}$ is injective for all $s\neq t\in S$.
\end{lemma}
\begin{proof}
	Let $s\neq t \in S$. If $m_{st} = 2$ there is nothing to show, as $X_{s, t} = X_s \times X_t$. Thus we assume $m_{st} \neq 2$ and let $1\neq g\in \ker( X_s \to X_{s, t} )$. Since $\ker(X_s \to \prod_{s\in J \in E(S)} \Aut(\Delta_J)) = \{1\}$, there exists $s \in K \in E(S)$ such that $g \notin \ker(X_s \to \Aut(\Delta_K))$. As $g\in \ker(X_s \to X_{s, t})$, we have $g \in \ker(X_s \to \Aut(\Delta_{\{s, t\}}))$ and hence $K \neq \{s, t\}$. Let $J = K \cup \{t\}$. Then $J$ is irreducible. As $\mathcal{F}_J$ is integrable, there exists a twin building $\Delta$ and a chamber $c$ of $\Delta$ such that $\mathcal{F}_J \cong \mathcal{F}(\Delta, c)$. By Lemma \ref{Lemma: integrable implies RGD-system} we obtain that $G_J$ acts on $\Delta$. As $g$ is trivial in $G_J$ but not in $\Aut(\Delta)$, this yields a contradiction and hence $X_s \to X_{s, t}$ is injective for all $s\neq t \in S$.
\end{proof}

\section{The $3$-spherical case}

In this section we let $\mathcal{F}$ be a Moufang foundation of irreducible $3$-spherical type $(W, S)$ and of rank at least $3$ satisfying Condition $\lco$. Moreover, we assume that for every irreducible $J \subseteq S$ with $\vert J \vert = 3$ the $J$-residue $\mathcal{F}_J$ is integrable and satisfies Condition $\lsco$. Let $\left( \Sigma_J \right)_{J \in E(S)}$ be an apartment of $\mathcal{F}$ and let $U_{\pm s}, X_s, X_{s, t}$ and $\mathcal{D}_{\mathcal{F}} = (G, (U_{\alpha})_{\alpha \in \Phi})$ be as before.

Our goal is to show that $\mathcal{F}$ is integrable. By Corollary \ref{inclusionsinjectivergd3} it suffices to show that the canonical mappings $U_{\pm s} \to G$ are injective and that $\mathcal{D}_{\mathcal{F}} = (G, (U_{\alpha})_{\alpha \in \Phi})$ satisfies (RGD$3$). We will show that $G$ acts non-trivially on a building and deduce both hypotheses from the action. For all $s\in S$ we fix $1 \neq e_s \in U_s$ and let $n_s := m(e_s)$.

For $s \neq t \in S$ with $m_{st} = 2$ we define $\Delta_{\{s, t\}}$ to be the spherical building associated with $X_s \times X_t$ (cf. \cite[Proposition $7.116$]{AB08}). Using \cite[Proposition $7.116$, Corollary $7.68$ and Remark $7.69$]{AB08}), we get (similar as for Moufang buildings) a natural labelling. In particular, $\left( (U_s \cup \infty_s)_{s\in S}, (\Delta_{\{s, t\}})_{s \neq t \in S} \right)$ is a blueprint. We will denote it by $\mathcal{B}_{\mathcal{F}}$ and remark that the restriction of $\mathcal{B}_{\mathcal{F}}$ to any $A_1 \times A_1$ residue is natural.

\begin{lemma}\label{integrablerealisable}
	Let $J \subseteq S$ be irreducible and of rank $3$. Then the blueprint $\mathcal{B}_{\mathcal{F}_J}$ is realisable.
\end{lemma}
\begin{proof}
	Since $\mathcal{F}_J$ is an integrable Moufang foundation, there exists a thick twin building $\Delta = (\Delta_+, \Delta_-, \delta_*)$ and a chamber $c$ of the twin building $\Delta$ such that $\mathcal{F}_J \cong \mathcal{F}(\Delta, c)$. Without loss of generality let $c\in \C_+$. Let $\Sigma$ be a twin apartment containing the images of the apartments $(\Sigma_K)_{K \in E(J)}$ of the foundation. We deduce from \cite[Theorem $1.5$]{MR95} and \cite[Theorem $8.27$ and Proposition $8.21$]{AB08} that $\Delta_+$ is a spherical Moufang building. By Proposition \ref{Ro89Proposition7.5} the building $\Delta_+$ conforms to the blueprint given by its restriction to $E_2(c)$ and the natural labelling of $\Delta_+$, i.e. $\left(U_{\alpha_s} \cup \{ \infty_s \})_{s\in J}, (R_{\{s, t\}}(c))_{s\neq t \in J} \right)$, where the $U_{\pm \alpha_s}$ are the root groups corresponding to the roots in $\Sigma \cap \Delta_+$. We will show that $\Delta_+$ conforms to $\mathcal{B}_{\mathcal{F}_J}$. As we have a labelling of $\Delta$ of type $(U_{\alpha_s} \cup \{ \infty_s \})_{s\in J}$ and $U_s \to U_{\alpha_s}$ are isomorphisms, we also have a labelling of type $(U_s \cup \{ \infty_s \})_{s\in J}$. Let $R$ be an $\{s, t\}$-residue of $\Delta_+$. Then there exists an isomorphism $\phi_R: R_{\{s, t\}}(c) \to R$ such that $x, \phi_R(x)$ have the same $s$- and $t$-labels for each $x\in R_{\{s, t\}}(c)$. As $\mathcal{F}_J \cong \mathcal{F}(\Delta, c)$, there exist isomorphisms $\alpha_K: \Delta_K \to R_K(c)$ for each $K \in E(J)$. Then $x$ and $\alpha_{\{s, t\}}(x)$ have the same $s$- and $t$-labels and hence $\Delta_+$ conforms to the blueprint $\mathcal{B}_{\mathcal{F}_J}$. In particular, $\mathcal{B}_{\mathcal{F}_J}$ is realisable.
\end{proof}

\begin{theorem}
	The blueprint $\mathcal{B_{\mathcal{F}}}$ is realisable.
\end{theorem}
\begin{proof}
	Let $J \subseteq S$ be of rank $3$. Then $J$ is spherical by assumption. If $J$ is irreducible, then $\mathcal{F}_J$ is integrable and hence $\mathcal{B}_{\mathcal{F}_J}$ is realisable by Lemma \ref{integrablerealisable}. If $J$ is reducible, then $\mathcal{B}_{\mathcal{F}_J}$ is realisable by Lemma \ref{reduciblerealisable}. Thus each restriction to a spherical rank $3$ subdiagram is realisable and hence the claim follows from Corollary \ref{Corollary: Ro89Theorem7.1+7.2}.
\end{proof}

Recall that $H_s = \langle m(u) m(v) \mid u, v \in U_s \backslash \{1\} \rangle$ and $B_J = \langle H_s, U_s \mid s \in J \rangle$ for $J \subseteq S$. Then we have $u_t^{h^{-1}} \in U_t$ for $u_t \in U_t, h \in H_s$ and $u_s^{n_t}, [u_s, u_t]^{n_t} \in B_{\{s, t\}}$ for $s \neq t \in S, u_s \in U_s, u_t \in U_t$.

\subsection*{An action associated with left multiplication}

\begin{theorem}\label{Theorem: 2-amalgam}
	Let $s\neq t \in S$. Then $B_{\{s, t\}}$ acts on $\Cbf(\mathcal{B}_{\mathcal{F}})$ as follows: Let $s_1, \ldots, s_k \in S$ be such that $s_1 \cdots s_k$ is reduced and let $u_i \in U_{s_i}$. Let $r\in \{s, t\}$ and $g\in U_r \cup H_r$. Then $\omega_B(g, ()) := ()$ and for $k>0$ we have
	\allowdisplaybreaks
	\begin{align*}
		\omega_B(g, (u_1, \ldots, u_k)) &:= \begin{cases}
			(u_1^{g^{-1}}, \omega_B(g^{n_{s_1}}, (u_2, \ldots, u_k))) & \text{if } g\in H_r, \\
			(gu_1, u_2, \ldots, u_k) & \text{if } r=s_1 \text{ and } g \in U_r, \\
			(u_1, \omega_B(g^{n_{s_1}} [g, u_1]^{n_{s_1}}, (u_2, \ldots, u_k))) & \text{if } r \neq s_1 \text{ and } g \in U_r.
		\end{cases}
	\end{align*}	
\end{theorem}
\begin{proof}
	We have to show that every relation in $B_{\{s, t\}}$ acts trivial on $(u_1, \ldots, u_k)$. We prove the hypothesis by induction on $k$. For $k=0$ there is nothing to show. Thus we assume $k \geq 1$. We consider a relation in $B_{\{s, t\}} = \langle H_s \cup H_t \rangle \ltimes \langle U_s \cup U_t \rangle$ (note that $H_s, H_t$ normalise $U_s, U_t$ and their intersection is trivial by \cite[Lemma $7.62$]{AB08}). Let $h_1, \ldots, h_m \in \langle H_s \cup H_t \rangle$ be such that $h_1 \cdots h_m = 1$ in $\langle H_s \cup H_t \rangle$. Then $\omega_B(h_1 \cdots h_m, (u_1, \ldots, u_k)) = (u_1^{ (h_1 \cdots h_m)^{-1} }, \omega_B( \prod_{i=1}^{m} h_i^{n_{s_1}}, (u_2, \ldots, u_k) ))$. We have $u_1^{ (h_1 \cdots h_m)^{-1} } = u_1$ and $\prod_{i=1}^{m} h_i^{n_{s_1}} = ( h_1 \cdots h_m )^{n_{s_1}}$ is a relation in $G_{\{s, t, s_1\}}$. As the product is contained in $B_{\{s, t, s_1\}}$, it is a relation in $B_{\{s, t, s_1\}}$. Note that $G_{\{s, t, s_1\}}$ is an RGD-system by Lemma \ref{Lemma: integrable implies RGD-system} or Lemma \ref{Lemma: reducible implies RGD-system}. Using Lemma \ref{Uplus2amalgam} we know that $B_{\{s, t, s_1\}}$ is the $2$-amalgam of the groups $B_{\{r\}}$ with $r\in \{s, t, s_1\}$. By the universal property of direct limits and induction, we deduce that $B_{\{s, t, s_1\}}$ acts on sequences of length at most $k-1$ and hence we can write
	\begin{align*}
		\prod_{i=1}^{m} h_i^{n_{s_1}} &= \prod_{i=1}^{n} g_i^{-1} r_i g_i
	\end{align*}
	for some $g_i \in B_{\{s, t, s_1\}}$ and relations $r_i \in B_{\{s, t\}} \cup B_{\{s, s_1\}} \cup B_{\{t, s_1\}}$. The claim follows now by induction. Let $v_1, \ldots, v_m \in U_s$ and $w_1, \ldots, w_m \in U_t$ be such that $\prod_{i=1}^{m} v_i w_i = 1$ in $\langle U_s \cup U_t \rangle$. We distinguish the following cases:
	\begin{enumerate}[label=(Case \Roman*)]
		\item $s_1 \notin \{ s, t \}$: Then $\omega_B( \prod_{i=1}^{m} v_i w_i, (u_1, \ldots, u_k) ) = (u_1, \omega_B( b, (u_2, \ldots, u_k) ))$, where $b = \prod_{i=1}^{m} v_i^{n_{s_1}} [v_i, u_1]^{n_{s_1}} w_i^{n_{s_1}} [w_i, u_1]^{n_{s_1}}$, and $b = \left( \prod_{i=1}^{m} v_i w_i \right)^{u_1 n_{s_1}}$ is a relation in $G_{\{s, t, s_1\}}$ and hence in $B_{\{s, t, s_1\}}$. As before, the claim follows by induction.
		
		\item $s_1 \in \{ s, t \}$: Without loss of generality we can assume $s_1 = t$. Then we compute $\omega_B( \prod_{i=1}^{m} v_i w_i, (u_1, \ldots, u_k) ) = ( w_1 \cdots w_m u_1, \omega_B( \prod_{i=1}^{m} v_i^{n_{s_1}} [ v_i, w_i \cdots w_m u_1 ]^{n_{s_1}} , (u_2, \ldots, u_k) ) )$. Since $\prod_{i=1}^{m} v_i w_i$ is a relation in $\langle U_s \cup U_t \rangle$ and $U_s$ acts trivial on the $t$-panel containing a fundamental chamber, the element $w_1 \cdots w_m \in U_t$ acts also trivial on this $t$-panel. Since the action is simply transitive, we deduce $w_1 \cdots w_m = 1$ and
		\allowdisplaybreaks
		\begin{align*}
			\prod_{i=1}^{m} v_i^{n_{s_1}} [ v_i, w_i \cdots w_m u_1 ]^{n_{s_1}} &= \left( \prod_{i=1}^{m} ( w_i \cdots w_m u_1 )^{-1} v_i w_i \cdots w_m u_1 \right)^{n_{s_1}} \\
			&= \left( \prod_{i=1}^{m} ( w_i \cdots w_m )^{-1} v_i w_i \cdots w_m \right)^{u_1 n_{s_1}} \\
			&= \left( (w_1 \cdots w_m)^{-1} \prod_{i=1}^{m} v_i w_i \right)^{u_1 n_{s_1}}
		\end{align*}
		is a relation in $G_{\{s, t\}}$ and hence in $B_{\{s, t\}}$. The claim follows by induction.
	\end{enumerate}
	It remains to show that for all $h\in H_t, v\in U_s, v' := h^{-1} v^{-1}h = (h^{-1}vh)^{-1} \in U_s$ the element $h^{-1} vh v'$ acts trivial on $(u_1, \ldots, u_k)$ (the case where $s, t$ are interchanged follows similarly). We distinguish the following cases:
	\begin{enumerate}[label=(Case \Roman*)]
		\item $s=s_1$: Then $\omega_B(h^{-1}vh v', (u_1, \ldots, u_k)) = \left( (v( v'u_1 )^{h^{-1}})^h, \omega_B( (h^{-1})^{n_{s_1}} h^{n_{s_1}}, (u_2, \ldots, u_k) ) \right)$. As $(v( v'u_1 )^{h^{-1}})^h = v^h v' u_1 = u_1$ in $U_s$ and $(h^{-1})^{n_{s_1}} h^{n_{s_1}}$ is a relation in $G_{\{s\}}$ and hence in $B_{\{s\}}$, the claim follows by induction.
		
		\item $s \neq s_1$: Then $\omega_B( h^{-1}vhv', (u_1, \ldots, u_k) ) = ( (u_1^{h^{-1}})^h, \omega_B( b, (u_2, \ldots, u_k)) )$, where $b = (h^{-1})^{n_{s_1}} v^{n_{s_1}} [v, u_1^{h^{-1}}]^{n_{s_1}} h^{n_{s_1}} (v')^{n_{s_1}} [v', u_1]^{n_{s_1}}$, and $b = ( h^{-1} v h v' )^{u_1 n_{s_1}}$ is a relation in $G_{\{s, t, s_1\}}$ and hence in $B_{\{s, t, s_1\}}$. As before, the claim follows by induction. \qedhere
	\end{enumerate}
\end{proof}

\begin{lemma}\label{Lemma: action stops}
	Let $J \subseteq S$ be such that $\vert J \vert \leq 3$. Let $u_i, u_i' \in \bigcup_{j\in J} U_j$.
	\begin{enumerate}[label=(\alph*)]
		\item Let $g\in B_J$ be such that $\omega_B(g, (u_1, \ldots, u_k)) = (u_1', \ldots, u_k')$. Then $g\cdot u_1 n_1 \cdots u_k n_kB_J = u_1' n_1 \cdots u_k' n_kB_J$.
		
		\item Let $g_1, \ldots, g_m \in \bigcup_{j\in J} U_j$ be and let $g = g_1 \cdots g_m$. We assume that $\omega_B(g_1 \cdots g_m, (u_1, \ldots,u_k)) = (u_1', \ldots, u_r', \omega_B(g', (u_{r+1}, \ldots, u_k)))$, i.e. $\omega_B$ does not stop after $r$ steps. Then we have $g' = (u_1' n_1 \cdots u_r' n_r)^{-1} g (u_1 n_1 \cdots u_r n_r)$ in $G_J$.
	\end{enumerate}
\end{lemma}
\begin{proof}
	We prove Assertion $(a)$ by induction on $k$. For $k=0$ the claim follows directly, as $\omega_B(g, ()) = ()$ and $g\cdot B_J = B_J$. Thus we assume $k\geq 1$. It suffices to show the claim for $g\in U_s \cup H_s$ with $s\in J$. If $g\in H_s$, we have $\omega_B(g, (u_1, \ldots, u_k)) = (u_1^{g^{-1}}, \omega_B(g^{n_1}, (u_2, \ldots, u_k)))$. Let $\omega_B(g^{n_1}, (u_2, \ldots, u_k)) = (u_2', \ldots, u_k')$. Using induction we have $g^{n_1} \cdot u_2 n_2 \cdots u_k n_k B_J = u_2' n_2 \cdots u_k' n_k B_J$. This implies
	\[ g \cdot u_1 n_1 \cdots u_kn_k B_J = u_1^{g^{-1}} n_1 g^{n_1} u_2 n_k \cdots u_k n_k B_J = u_1^{g^{-1}} n_1 u_2' n_2 \cdots u_k' n_k B_J \]
	As $\omega_B(g, (u_1, \ldots, u_k)) = (u_1^{g^{-1}}, u_2', \ldots, u_k')$, the claim follows. Assume $u_1 \in U_s$ and $g\in U_s$. Then $\omega_B(g, (u_1, \ldots, u_k)) = (gu_1, u_2, \ldots, u_k)$ and $g \cdot u_1 n_1 \cdots u_k n_k B_J = (gu_1) n_1 u_2 n_2 \cdots u_k n_k B_J$. Now we assume $u_1 \in U_t$ for some $s\neq t \in S$. Then $\omega_B(g, (u_1, \ldots, u_k)) = (u_1, \omega_B( g^{n_1} [g, u_1]^{n_1}, (u_2, \ldots, u_k) ))$. Let $\omega_B( g^{n_1} [g, u_1]^{n_1}, (u_2, \ldots, u_k) ) = (u_2', \ldots, u_k')$. Using induction we have $g^{n_1} [g, u_1]^{n_1} \cdot u_2 n_2 \cdots u_k n_k B_J = u_2'n_2 \cdots u_k'n_k B_J$. This implies
	\[ g u_1 n_1 \cdots u_k n_k B_J = u_1 n_1 g^{n_1} [g, u_1]^{n_1} u_2 n_2 \cdots u_k n_k B_J = u_1 n_1 u_2'n_2 \cdots u_k'n_k B_J \]
	As $\omega_B(g, (u_1, \ldots, u_k)) = (u_1, u_2', \ldots, u_k')$, the claim follows. We prove Assertion $(b)$ by induction on $r$. Let $j \in J$ be such that $u_1 \in U_j$. Let $i_1 < \ldots < i_n$ be all indices such that $g_{i_1}, \ldots, g_{i_n} \in  U_j$. Then we have the following:
	\[ \omega_B(g, (u_1, \ldots, u_k)) = (g_{i_1} \cdots g_{i_n} u_1, \omega_B(g', (u_2, \ldots, u_k))) \]
	where $g' = \prod_{x=0}^{n} \prod_{y=i_x +1}^{i_{x+1} -1} g_y^{n_j} [g_y, g_{i_{x+1}} \cdots g_{i_n} u_1]^{n_j}$ and $i_0 := 0, i_{n+1} := m+1$. We compute the following:
	\begin{align*}
		g' &= \prod_{x=0}^{n} \prod_{y=i_x +1}^{i_{x+1} -1} g_y^{n_j} [g_y, g_{i_{x+1}} \cdots g_{i_n} u_1]^{n_j} \\
		&= \prod_{x=0}^{n} \prod_{y=i_x +1}^{i_{x+1} -1} \left( (g_{i_{x+1}} \cdots g_{i_n} u_1)^{-1} g_y (g_{i_{x+1}} \cdots g_{i_n} u_1) \right)^{n_j} \\
		&= \prod_{x=0}^{n} \left( (g_{i_{x+1}} \cdots g_{i_n} u_1)^{-1} \left( \prod_{y=i_x +1}^{i_{x+1} -1} g_y \right) (g_{i_{x+1}} \cdots g_{i_n} u_1) \right)^{n_j} \\
		&= \left( (g_{i_1} \cdots g_{i_n} u_1)^{-1} \left( \prod_{y=1}^{m} g_m \right) u_1 \right)^{n_j} \\
		&= \left( g_{i_1} \cdots g_{i_n} u_1 n_j \right)^{-1} g u_1 n_1
	\end{align*}
	As $g' \in \langle U_j \mid j \in J \rangle$, the claim follows by induction.
\end{proof}

\begin{lemma}\label{Lemma: BJ homotopic sequences to homotopic sequences}
	The action in Theorem \ref{Theorem: 2-amalgam} maps equivalent sequences to equivalent sequences. In particular, $\omega_B$ extends to an action of the building $\Cbf_{\mathcal{B}_{\mathcal{F}}}$.
\end{lemma}
\begin{proof}
	In this proof we use the notation $U_{s_1 \cdots s_k}$ for $U_{s_1} \times \cdots \times U_{s_k}$. It suffices to show that every element in $\bigcup_{r\in S} U_r \cup H_r$ maps two elementary equivalent sequences to elementary equivalent sequences. Let $s\neq t \in S$, let $\bar{u} = \bar{u}_1 \bar{u}_0 \bar{u}_2$ be of type $(f_1, p(s, t), f_2)$ and $ \bar{v} = \bar{u}_1 \bar{v}_0 \bar{u}_2 $ be of type $(f_1, p(t, s), f_2)$, where $\bar{u}_0$ and $\bar{v}_0$ are equivalent in $\Delta_{\{s, t\}}$. It suffices to show the claim for the sequences $\bar{u}_0 \bar{u}_2$ of type $(p(s, t), f_2)$ and $\bar{v}_0 \bar{u}_2 $ of type $(p(t, s), f_2)$. Let $w\in W$ and $s \neq t \in S$ be such that $\ell(sw) = \ell(w) +1 = \ell(tw)$. Let $(u_1, \ldots, u_m) \in U_{p(s, t)}, (v_1, \ldots, v_m) \in U_{p(t, s)}$ be equivalent in $\Delta_{\{s, t\}}$, let $s_1, \ldots, s_k \in S$ be such that $w = s_1 \cdots s_k$ is reduced and let $(w_1, \ldots, w_k) \in U_{s_1 \cdots s_k}$. Then $\left( u_1, \ldots, u_m, w_1, \ldots, w_k \right)$ and $\left( v_1, \ldots, v_m, w_1, \ldots, w_k \right)$ are elementary equivalent. First we assume that $g \in H_r$ for some $r\in S$. Then we have the following:
	\allowdisplaybreaks
	\begin{align*}
		\omega_B(g, (u_1, \ldots, u_m, w_1, \ldots, w_k)) &= (u_1', \ldots, u_m', \omega_B(g^{p(n_s, n_t)}, (w_1, \ldots, w_k))) \\
		\omega_B(g, (v_1, \ldots, v_m, w_1, \ldots, w_k)) &= (v_1', \ldots, v_m', \omega_B(g^{p(n_t, n_s)}, (w_1, \ldots, w_k)))
	\end{align*}
	Let $J := \{r, s, t\}$. We deduce from Lemma \ref{Lemma: action stops}$(a)$ and Corollary \ref{Homotopyinrank2} that
	\allowdisplaybreaks
	\begin{align*}
		u_1' n_1 \cdots u_m' n_m B_J &= g \cdot u_1n_1 \cdots u_mn_mB_J = g \cdot v_1n_1' \cdots v_mn_m'B_J = v_1' n_1' \cdots v_m'n_m' B_J
	\end{align*}
	Using Corollary \ref{Homotopyinrank2} again, we deduce that $(u_1', \ldots, u_m') \simeq (v_1', \ldots, v_m')$. Since $p(n_s, n_t) = p(n_t, n_s)$ is a relation in $G_{\{r, s, t\}}$, we obtain $g^{p(n_s, n_t)} = g^{p(n_t, n_s)}$ in $B_{\{r, s, t\}}$ and hence they act equally on $(w_1, \ldots, w_k)$.
	
	Now let $g\in U_r$ for some $r\in S$. We note that $(u_1, \ldots, u_k) \simeq (v_1, \ldots v_k)$ for $u_i, v_i \in U_s \cup U_t$ implies $u_1 n_1 \cdots u_k n_k = v_1 n_1' \cdots v_kn_k'$ (cf. Corollary \ref{Homotopyinrank2} and \cite[p. $87$]{Ro89}). We consider the following two cases:
	\begin{enumerate}[label=(Case \Roman*)]
		\item $r\in \{s, t\}$: W.l.o.g. we assume that $r=s$. Then
		\begin{align*}
			\omega_B(g, (u_1, \ldots, u_m, w_1, \ldots, w_k)) = (gu_1, u_2, \ldots, u_m, w_1, \ldots, w_k).
		\end{align*}
		Let $v_1', \ldots, v_m'$ be such that $\omega_B(g, (v_1, \ldots, v_k)) = (v_1', \ldots, v_k')$. By Lemma \ref{Lemma: action stops}$(a)$ and Corollary \ref{Homotopyinrank2} we deduce $(gu_1, u_2, \ldots, u_k) \simeq (v_1', \ldots, v_k')$. Now there are two possibilities:
		\begin{enumerate}
			\item $\omega_B(g, (v_1, \ldots, v_m, w_1, \ldots, w_k)) = (v_1', \ldots, v_m', w_1, \ldots, w_k)$, i.e. the action stops after at most $m$ steps. Then the claim follows directly.
			
			\item $\omega_B(g, (v_1, \ldots, v_k, w_1, \ldots, w_k)) = (v_1', \ldots, v_m', \omega_B(g', w_1, \ldots, w_k))$: Then by Lemma \ref{Lemma: action stops}$(b)$ we deduce $g' = (v_1' n_1' \cdots v_m' n_m')^{-1} g (v_1 n_1' \cdots v_m n_m')$. But then we infer the following in $G_{\{s, t\}}$:
			\begin{align*}
				g' &= (v_1' n_1' \cdots v_m' n_m')^{-1} g (v_1 n_1' \cdots v_m n_m') \\
				&= (gu_1 n_1 u_2 n_2 \cdots u_m n_m)^{-1} g (u_1 n_1 \cdots u_k n_k) \\
				&= 1
			\end{align*}
			Thus $\omega_B(g', (w_1, \ldots, w_k)) = (w_1, \ldots, w_k)$ and the claim follows.
		\end{enumerate}
		
		\item $r \notin \{s, t\}$: Let $\omega_B(g, (u_1, \ldots, u_k)) = (u_1', \ldots, u_k')$ and $\omega_B(g, (v_1, \ldots, v_k)) = (v_1', \ldots, v_k')$ be. Then Lemma \ref{Lemma: action stops}$(a)$ implies
		\allowdisplaybreaks
		\begin{align*}
			g \cdot u_1 n_1 \dots u_m n_m B_{\{r, s, t\}} = u_1' n_1 \cdots u_n' n_m B_{\{r, s, t\}} \\
			g \cdot v_1 n_1' \cdots v_m n_m' B_{\{r, s, t\}} = v_1' n_1' \cdots v_m' n_m' B_{\{r, s, t\}}
		\end{align*}
		In particular, we deduce that
		\allowdisplaybreaks
		\begin{align*}
			g' &:= \left( u_1' n_1 \cdots u_n' n_m \right)^{-1} g u_1 n_1 \dots u_m n_m \in B_{\{r, s, t\}} \\
			g'' &:= \left( v_1' n_1' \cdots v_m' n_m' \right)^{-1} g v_1 n_1' \cdots v_m n_m' \in B_{\{r, s, t\}}
		\end{align*}
		Assume that $\omega_B(g, (u_1, \ldots, u_m, w_1, \ldots, w_k)) = (u_1', \ldots, u_m', w_1, \ldots, w_k)$, i.e. the action stops after at most $m$ steps. Then $gu_1 n_1 \cdots u_mn_m = u_1' n_1 \cdots u_m' n_m$ and $g'$ is a relation in $G_{\{r, s, t\}}$. We will show that this is a contradiction. Let $u, u' \in \langle U_{\alpha_s} \cup U_{\alpha_t} \rangle$ be such that $u_1 n_1 \cdots u_m n_m = u n(w)$ and $u_1'n_1 \cdots u_m' n_m = u' n(w)$ as in \cite[Lemma $(7.4)$]{Ro89}. Then $(u')^{-1}g u$ is a relation as well and we deduce $U_{\alpha_r} \cap \langle U_{\alpha_s} \cup U_{\alpha_t} \rangle \neq \{1\}$, which is a contradiction. Thus $g'$ cannot be a relation. Moreover, we have $u_1 n_1 \cdots u_m n_m = v_1 n_1' \cdots v_m n_m'$, as $(u_1, \ldots, u_m) \simeq (v_1, \ldots, v_m)$, and $u_1' n_1 \cdots u_m' n_m = v_1' n_1' \cdots v_m' n_m'$, as $(u_1', \ldots, u_m') \simeq (v_1', \ldots, v_m')$ by Corollary \ref{Homotopyinrank2}. In particular, we have $g' = g''$ in $B_{\{r, s, t\}}$ and $g''$ is no relation as well. In particular, the action does not stop after at most $m$ steps. Together with Theorem \ref{Theorem: 2-amalgam} we obtain the following:
		\allowdisplaybreaks
		\begin{align*}
			\omega_B(g, (u_1, \ldots, u_m, w_1, \ldots, w_k)) &= (u_1', \ldots, u_m', \omega_B(g', (w_1, \ldots, w_m))) \\
			&\simeq (v_1', \ldots, v_m', \omega_B(g'', (w_1, \ldots, w_m))) \\
			&= \omega_B(g, (v_1, \ldots, v_m, w_1, \ldots, w_k)). \qedhere
		\end{align*}
	\end{enumerate}
\end{proof}

\begin{lemma}
	Let $s_1, \ldots, s_k \in S$ be such that $s_1 \cdots s_k$ is reduced and let $u_i \in U_{s_i}$. Then we have for each $s\in S$ a well-defined mapping 
	\[ \omega(n_s, \cdot): \Cbf(\mathcal{B}_{\mathcal{F}}) \to \Cbf_{\mathcal{B}_{\mathcal{F}}}, (u_1, \ldots, u_k) \mapsto \begin{cases}
		\omega_B(n_{s_1}^2, [u_2, \ldots, u_k]) & \text{if } s=s_1, u_1 = 1, \\
		[\overline{u}_1, \omega_B(b(u_1), [u_2, \ldots, u_k])] & \text{if } s=s_1, u_1 \neq 1, \\
		[1_{U_s}, u_1, \ldots, u_k] & \text{if } \ell(ss_1 \cdots s_k) = k+1 \\
		\omega( g, (v_1, \ldots, v_k) ) & \text{else,}
	\end{cases} \]
	where $(v_1, \ldots, v_k)$ is of type $(s, t_2, \ldots, t_k), st_2 \cdots t_k = s_1 \cdots s_k$ and $(u_1, \ldots, u_k) \simeq (v_1, \ldots, v_k)$. Moreover, this mapping extends to $\Cbf_{\mathcal{B}_{\mathcal{F}}}$. More precisely, for any two equivalent sequences $(u_1, \ldots, u_k) \simeq (v_1, \ldots, v_k)$ we have $\omega(n_s, (u_1, \ldots, u_k)) = \omega(n_s, (v_1, \ldots, v_k))$.
\end{lemma}
\begin{proof}
	To show that $\omega(n_s, \cdot)$ is well-defined, we have to show that the case $\ell(ss_1 \cdots s_k) = k-1$ and $s \neq s_1$ is independent of the choice of $(v_1, \ldots, v_k)$. Therefore, we let $r_2, \ldots, r_k \in S, w_1 \in U_s$ and $w_i \in U_{r_i}$ be such that $s t_2 \cdots t_k = s_1 \cdots s_k = sr_2 \cdots r_k$ and
	\[ (v_1, \ldots, v_k) \simeq (u_1, \ldots, u_k) \simeq (w_1, \ldots, w_k) \]
	As $t_2 \cdots t_k = r_2 \cdots r_k$, there exist $v_i' \in U_{t_i}$ such that $(w_2, \ldots, w_k) \simeq (v_2', \ldots, v_k')$ and hence $(w_1, \ldots, w_k) \simeq (w_1, v_2', \ldots, v_k')$.	By Corollary \ref{Corollary: Ro89Theorem7.1+7.2} we have $v_1 = w_1, v_i = v_i'$ and hence $(v_2, \ldots, v_k) \simeq (w_2, \ldots, w_k)$. Using Lemma \ref{Lemma: BJ homotopic sequences to homotopic sequences} we deduce
	\begin{align*}
		\omega(n_s, (v_1, \ldots, v_k)) &= \begin{cases}
			\omega_B( n_{s_1}^2, [v_2, \ldots, v_k] ) & \text{if } v_1 = 1 \\
			[\overline{v}_1, \omega_B(b(v_1), [v_2, \ldots, v_k])] & \text{if } v_1 \neq 1
		\end{cases} \\
		&= \begin{cases}
			\omega_B( n_{s_1}^2, [w_2, \ldots, w_k] ) & \text{if } w_1 = 1 \\
			[\overline{w}_1, \omega_B(b(w_1), [w_2, \ldots, w_k])] & \text{if } w_1 \neq 1
		\end{cases} \\
		&= \omega(n_s, (w_1, \ldots, w_k))
	\end{align*}
	Thus the mapping is well-defined. Suppose $s_1, \ldots, s_k, t_1, \ldots, t_k \in S$ such that $s_1 \cdots s_k = t_1 \cdots t_k$ is reduced. Let $u_i \in U_{s_i}, v_i \in U_{t_i}$ be such that $(u_1, \ldots, u_k) \simeq (v_1, \ldots, v_k)$. If $\ell(ss_1 \cdots s_k) = k+1$, we have $(1, u_1, \ldots, u_k) \simeq (1, v_1, \ldots, v_k)$ and hence $[1, u_1, \ldots, u_k] = [1, v_1, \ldots, v_k]$. Now we assume $\ell(ss_1 \cdots s_k) = k-1$. Then there exist $s_2', \ldots, s_k', t_2', \ldots, t_k' \in S$ such that $s_1 \cdots s_k = ss_2' \cdots s_k'$ and $t_1 \cdots t_k = st_2' \cdots t_k'$. Let $u_1', v_1' \in U_s, u_i' \in U_{s_i'}, v_i' \in U_{t_i'}$ be such that $(u_1', \ldots, u_k') \simeq (u_1, \ldots, u_k) \simeq (v_1, \ldots, v_k) \simeq (v_1', \ldots, v_k')$. As before we deduce $u_1' = v_1'$ and $(u_2', \ldots, u_k') \simeq (v_2', \ldots, v_k')$. The claim follows now from Lemma \ref{Lemma: BJ homotopic sequences to homotopic sequences}.
\end{proof}

\begin{lemma}\label{Lemma: ns automorphism}
	For each $s\in S$ we have $n_s \in \Aut(\Cbf_{\mathcal{B}_{\mathcal{F}}})$ via
		\begin{align*}
		\omega(g, [u_1, \ldots, u_k]) &:= \begin{cases}
			\omega(n_s, \omega_B(n_s^{-2}, [u_1, \ldots, u_k])) & \text{if } g = n_s^{-1} \\
			\omega(n_s, [u_1, \ldots, u_k]) & \text{if } g = n_s
		\end{cases}
	\end{align*}
\end{lemma}
\begin{proof}
	It suffices to show that $n_s^{-1} n_s$ acts trivial on every sequence. Let $s_1, \ldots, s_k \in S$ be such that $s_1 \cdots s_k$ is reduced and let $u_i \in U_{s_i}$. We distinguish the following cases:
	\begin{enumerate}[label=(\roman*)]
		\item $s=s_1$ and $u_1 = 1$: Then $\omega(n_s^{-1} n_s, [u_1, \ldots, u_k]) = [1, \omega_B( n_s^{-2} n_s^2, [u_2, \ldots, u_k] )] =[u_1, \ldots, u_k]$.
		
		\item $s=s_1$ and $u_1 \neq 1$: Then
		\begin{align*}
			\omega(n_s^{-1} n_s, [u_1, \ldots, u_k]) &= \omega( n_s^{-1}, [\overline{u}_1, \omega_B(b(u_1), [u_2, \ldots, u_k])] ) \\
			&= \omega(n_s, [\overline{u}_1^{n_s^2}, \omega_B( (n_s^{-2})^{n_s} b(u_1), [u_2, \ldots, u_k] )]) \\
			&= [ \overline{\overline{u}_1^{n_s^{2}}}, \omega_B( b(\overline{u}_1^{n_s^{2}}) (n_s^{-2})^{n_s} b(u_1) ), [u_2, \ldots, u_k] ]
		\end{align*} 
		Note that $u_1 n_s B_{\{s\}} = n_s^{-1} n_s u_1 n_s B_{\{s\}} = n_s n_s^{-2} \overline{u}_1 n_s b(u_1) B_{\{s\}} = n_s \overline{u}_1^{n_s^2} n_s B_{\{s\}} = \overline{\overline{u}_1^{n_s^2}} n_s B_{\{s\}}$. Thus $u_1 = \overline{\overline{u}_1^{n_s^{2}}}$ and $b(\overline{u}_1^{n_s^{2}}) (n_s^{-2})^{n_s}b(u_1)$ is a relation in $B_{\{s\}}$. The claim follows now from Theorem \ref{Theorem: 2-amalgam}.
		
		\item $\ell(s s_1 \cdots s_k) = k+1$: Then $\omega(n_s^{-1} n_s, [u_1, \ldots, u_k]) = \omega_B( n_s^{2} (n_s^{-2})^{n_s}, [u_1, \ldots, u_k] ) = [u_1, \ldots, u_k]$.
		
		\item $\ell(ss_1 \cdots s_k) = k-1$ and $s\neq s_1$: Let $(v_1, \ldots, v_k) \in [u_1, \ldots, u_k]$ with $v_1 \in U_s$. Then $\omega(n_s^{-1} n_s, [v_1, \ldots, v_k]) = [v_1, \ldots, v_k]$ as before and hence $\omega(n_s^{-1} n_s, [u_1, \ldots, u_k]) = [v_1, \ldots, v_k] = [u_1, \ldots, u_k]$. \qedhere
	\end{enumerate}
\end{proof}

\subsection*{Some relations in $\Aut(\Cbf_{\mathcal{B}_{\mathcal{F}}})$}

In this subsection we will show that the relations in Theorem \ref{presentationofanRGDsystem} act trivial on $\Cbf_{\mathcal{B}_{\mathcal{F}}}$. This will imply that $G$ acts on the building $\Cbf_{\mathcal{B}_{\mathcal{F}}}$.

\begin{lemma}\label{Psacts}
	For $s\in S$ and $1 \neq u_s \in U_s$ We have $n_s u_s n_s = \overline{u}_s n_s b(u_s)$ in $\Aut(\Cbf_{\mathcal{B}_{\mathcal{F}}})$.
\end{lemma}
\begin{proof}
	Let $(u_1, \ldots, u_k)$ be a sequence of type $(s_1, \ldots, s_k)$. We distinguish the following two cases:
	\begin{enumerate}[label=(\alph*)]
		\item $\ell(ss_1 \cdots s_k) = k+1$. Then we have
		\begin{align*}
			\omega( n_s u_s n_s, [u_1, \ldots, u_k] ) = [\overline{u}_s, \omega_B( b(u_s), [u_1, \ldots, u_k] )] = \omega( \overline{u}_s n_s b(u_s), [u_1, \ldots, u_k] )
		\end{align*}
	
		\item $\ell(ss_1 \cdots s_k) = k-1$: W.l.o.g. we assume $s = s_1$. We let $b(u_s) = v_s h_s$ and distinguish the following cases:
		\begin{enumerate}[label=(\roman*)]
			\item $u_1 = 1$: Then we compute the following:
			\begin{align*}
				\omega(n_s u_s n_s, [1, u_2, \ldots, u_k]) &= [1, \omega_B(u_s n_s^2, [u_2, \ldots, u_k])] \\
				\omega(\overline{u}_s n_s b(u_s), [1, u_2, \ldots, u_k]) &= [ \overline{u}_s \overline{v}_s, \omega_B( b(v_s) h_s^{n_s}, [u_2, \ldots, u_k] ) ]
			\end{align*}
			Comparing the cosets in the Moufang building we obtain $\overline{u}_s \overline{v}_s = 1$. Moreover, we have the following:
			\begin{align*}
				b(v_s) h_s^{n_s} = n_s^{-1} \overline{v}_s^{-1} \overline{v}_s n_s b(v_s) n_s^{-1} h_s n_s = n_s^{-1} \overline{v}_s^{-1} \overline{u}_s^{-1} \overline{u}_s n_s v_s n_s n_s^{-1} h_s n_s = n_s^{-1} n_s u_s n_s n_s = u_s n_s^2
			\end{align*}
			Thus $( u_s n_s^2 )^{-1} b(v_s) h_s^{n_s}$ is a relation in $B_{\{s\}}$ and the claim follows.
			
			\item $u_1 \neq 1$: Then we compute the following:
			\begin{align*}
				\omega(n_s u_s n_s, [u_1, \ldots, u_k]) &= \begin{cases}
					\omega_B( n_s^2 b(u_1), [u_2, \ldots, u_k] ) & \text{if } u_s\overline{u}_1 = 1 \\
					[ \overline{u_s \overline{u}_1}, \omega_B( b(u_s \overline{u}_1) b(u_1), [u_2, \ldots, u_k] ) ] & \text{if } u_s \overline{u}_1 \neq 1
				\end{cases} \\
			\omega(\overline{u}_s n_s b(u_s), [u_1, \ldots, u_k]) &= \begin{cases}
				\omega_B( \overline{u}_s n_s^2 h_s^{n_s}, [u_2, \ldots, u_k] ) & \text{if } v_s u_1^{h_s^{-1}} = 1 \\
				[ \overline{u}_s \overline{v_s u_1^{h_s^{-1}}}, \omega_B( b(v_s u_1^{h_s^{-1}}) h_s^{n_s}, [u_2, \ldots, u_k] ) ] & \text{if } v_s u_1^{h_s^{-1}} \neq 1
			\end{cases}
			\end{align*}
			Note that in $X_s$ we have $n_s u_s \overline{u}_1 n_s b(u_1) = n_s u_s n_s u_1 n_s = \overline{u}_s n_s b(u_s) u_1 n_s = \overline{u}_s n_s v_s u_1^{h_s^{-1}} n_s h_s^{n_s}$. If $u_s \overline{u}_1 = 1$, then the left hand side is contained in $B_{\{s\}}$. If $v_s u_1^{h_s^{-1}} \neq 1$, then $U_{-\alpha_s} \cap B_{\{s\}} \neq \{1\}$, which is a contradiction. Thus $u_s \overline{u}_1 = 1$ implies $v_s u_1^{h_s^{-1}} = 1$. On the other hand if $v_s u_1^{h_s^{-1}} = 1$, then the right hand side is contained in $B_{\{s\}}$. If $u_s \overline{u}_1 \neq 1$, then again $U_{-\alpha_s} \cap B_{\{s\}} \neq \{1\}$, which is a contradiction. Thus $v_s u_1^{h_s^{-1}} = 1$ implies $u_s \overline{u}_1 = 1$.
			
			Assume that $u_s \overline{u}_1 = 1 = v_s u_1^{h_s^{-1}}$. Then $n_s^2 b(u_1) = \overline{u}_s n_s^2 h_s^{n_s}$ in $X_s$ and hence in $B_{\{s\}}$. Thus we are done. Now we assume that $u_s \overline{u}_1 \neq 1 \neq v_s u_1^{h_s^{-1}}$. Note that in $X_s$ we have the following:
			\begin{align*}
				\overline{u_s \overline{u}_1} n_s b(u_s \overline{u}_1) b(u_1) &= n_s u_s \overline{u}_1 n_s b(u_1) \\
				&= n_s u_s n_s u_1 n_s \\
				&= \overline{u}_s n_s b(u_s) u_1 n_s \\
				&= \overline{u}_s n_s v_s u_1^{h_s^{-1}} n_s h_s^{n_s} \\
				&= \overline{u}_s \overline{v_s u_1^{h_s^{-1}}} n_s b(v_s u_1^{h_s^{-1}}) h_s^{n_s}
			\end{align*}
			This implies $\overline{u_s \overline{u}_1} = \overline{u}_s \overline{v_s u_1^{h_s^{-1}}}$ and hence $b(u_s \overline{u}_1) b(u_1) = b(v_s u_1^{h_s^{-1}}) h_s^{n_s}$. Thus we are done. \qedhere
		\end{enumerate}
	\end{enumerate}
\end{proof}

\begin{lemma}
	For $s, t \in S$ and $1 \neq h \in H_t$ we have $n_s h n_s = n_s^2 h^{n_s}$ in $\Aut(\Cbf_{\mathcal{B}_{\mathcal{F}}})$.
\end{lemma}
\begin{proof}
	Let $(u_1, \ldots, u_k)$ be a sequence of type $(s_1, \ldots, s_k)$. We distinguish the following cases:
	\begin{enumerate}[label=(\alph*)]
		\item $\ell(ss_1 \cdots s_k) = k+1$: Then we have $\omega( n_s h n_s, [u_1, \ldots, u_k] ) = \omega_B( n_s^2 h^{n_s}, [u_1, \ldots, u_k])$.
	
		\item $\ell(ss_1 \cdots s_k) = k-1$: Again we can assume $s=s_1$. We distinguish the cases $u_1 =1$ and $u_1 \neq 1$ and compute the following:
		\begin{align*}
			\omega(n_s h n_s, [1, u_2 \ldots, u_k]) &= [1, \omega_B( h n_s^2, [u_2, \ldots, u_k] )] \\
			\omega( n_s^2 h^{n_s}, [1, u_2, \ldots, u_k] ) &= [1, \omega_B( (n_s^2)^{n_s} h^{n_s^2}, [u_2, \ldots, u_k] )] \\
			\omega( n_s h n_s, [u_1, \ldots, u_k] ) &= [ \overline{\overline{u}_1^{h^{-1}}}, \omega_B( b(\overline{u}_1^{h^{-1}}) h^{n_s}b(u_1), [u_2, \ldots, u_k] ) ] \\
			\omega( n_s^2 h^{n_s}, [u_1, \ldots, u_k] ) &= [ (u_1^{(h^{n_s})^{-1}})^{n_s^{-2}}, \omega_B( (n_s^2)^{n_s} h^{n_s^2}, [u_2, \ldots, u_k] )]
		\end{align*}
		In the case $u_1 = 1$ we have $h n_s^2 = n_s^2 n_s^{-2} h n_s^2 = n_s^2 h^{n_s^2}$ and hence we are done. In the case $u_1 \neq 1$ we have $\overline{\overline{u}_1^{h^{-1}}} = (u_1^{(h^{n_s})^{-1}})^{n_s^{-2}}$, as $\overline{\overline{u}_1^{h^{-1}}} n_s B_{\{s\}} = (u_1^{(h^{n_s})^{-1}})^{n_s^{-2}} n_s B_{\{s\}}$ and hence $b(\overline{u}_1^{h^{-1}}) h^{n_s}b(u_1) = (n_s^2)^{n_s} h^{n_s^2}$. The claim follows now from Theorem \ref{Theorem: 2-amalgam}. \qedhere
	\end{enumerate}
\end{proof}

\begin{lemma}\label{Lemma: relation conjugate root groups}
	For $s\neq t\in S$ and $\alpha_s \neq \alpha \in \Phi_+^{\{s, t\}}$ we have $n_s u_{\alpha} n_s = n_s^2 u_{\alpha}^{n_s}$ in $\Aut(\Cbf_{\mathcal{B}_{\mathcal{F}}})$.
\end{lemma}
\begin{proof}
	Let $\alpha_s \neq \alpha \in \Phi_+^{\{s, t\}}$. By Lemma \ref{coUplus} there exists $g_i \in U_t,h_i \in U_s$ such that $u_{\alpha} = \prod_{i=1}^{n} g_i h_i$. Note that $h_1 \cdots h_n = 1$ as in the proof of Theorem \ref{Theorem: 2-amalgam}. We distinguish the following cases:
	\begin{enumerate}[label=(\alph*)]
		\item $\ell(ss_1 \cdots s_k) = k+1$: We deduce the following
		\begin{align*}
			\omega(n_s u_{\alpha} n_s, [u_1, \ldots, u_k]) &= \omega(n_s, [h_1 \cdots h_n, \omega_B( b, [u_1, \ldots, u_k] )]) = \omega_B(n_s^2 b, [u_1, \ldots, u_k])
		\end{align*}
		where $b = \prod_{i=1}^{n} g_i^{n_s} [g_i, h_i \cdots h_n]^{n_s}$. Note that $b = \left( (h_1 \cdots h_n)^{-1} \prod_{i=1}^{n} g_ih_i \right)^{n_s} = u_{\alpha}^{n_s}$ and hence the claim follows by Lemma \ref{Lemma: BJ homotopic sequences to homotopic sequences}.
		
		\item $\ell(ss_1 \cdots s_k) = k-1$. W.l.o.g. we assume $s=s_1$. Let $g_i' \in U_t, h_i' \in U_s$ such that $u_{\alpha}^{n_s} = \prod_{i=1}^{m} g_i'h_i'$. Again, $h_1' \cdots h_m' = 1$. We deduce the following:
		\begin{enumerate}[label=(\roman*)]
			\item If $u_1 = 1$, we have:
			\allowdisplaybreaks
			\begin{align*}
				\omega(n_s u_{\alpha} n_s, [u_1, \ldots, u_k]) &= [1, \omega_B( u_{\alpha} n_s^{2}, [u_2, \ldots, u_k] )] \\
				\omega(n_s^2 u_{s\alpha}, [u_1, \ldots, u_k]) &= \omega(n_s^{2}, [h_1' \cdots h_m', \omega_B(b, [u_2, \ldots, u_k])]) \\
				&= [1, \omega_B( (n_s^{2})^{n_s} b, [u_2, \ldots, u_k] )]
			\end{align*}
			where $b = \prod_{i=1}^{m} (g_i')^{n_s} [g_i', h_i' \cdots h_m']^{n_s}$. Note that $b = ( (h_1' \cdots h_m')^{-1} \prod_{i=1}^m g_i' h_i' )^{n_s} = (u_{\alpha}^{n_s})^{n_s} = n_s^{-2} u_{\alpha} n_s^2$. Thus $(n_s^2)^{n_s} b = u_{\alpha} n_s^2$ in $B_{\{s, t\}}$ and the claim follows from Lemma \ref{Lemma: BJ homotopic sequences to homotopic sequences}.
				
			\item If $u_1 \neq 1$, we compute the following:
			\allowdisplaybreaks
			\begin{align*}
				\omega( n_s u_{\alpha} n_s, [u_1, \ldots, u_k] ) &= \omega( n_s u_{\alpha}, [\overline{u}_1, \omega_B( b(u_1), [u_2, \ldots, u_k] )] ) \\
				&= \omega(n_s, [h_1 \cdots h_n\overline{u}_1, \omega_B(b\cdot b(u_1), [u_2, \ldots, u_k])]) \\
				&= [ \overline{\overline{u}}_1, \omega_B( b(\overline{u}_1) \cdot b \cdot b(u_1), [u_2, \ldots, u_k] ) ] \\
				\omega(n_s^2 u_{s\alpha}, [u_1, \ldots, u_k]) &= \omega( n_s^2, [ h_1' \cdots h_m' u_1, \omega_B( b', [u_2, \ldots, u_k] ) ]) \\
				&= [ u_1^{n_s^{-2}}, \omega_B( (n_s^2)^{n_s}b', [u_2, \ldots, u_k] ) ]
			\end{align*}
			where $b = \prod_{i=1}^{n} g_i^{n_s} [g_i, h_i \cdots h_m u_1]^{n_s}$ and $b' = \prod_{i=1}^{m} g_i'^{n_s} [g_i', h_i' \cdots h_m' u_1]^{n_s}$. As before, we have $\overline{\overline{u}}_1 = u_1^{n_s^{-2}}$ and $b(\overline{u}_1)\cdot b \cdot b(u_1) = (n_s^2)^{n_s} b'$ in $B_{\{s\}}$. Now the claim follows from Lemma \ref{Lemma: BJ homotopic sequences to homotopic sequences}. \qedhere
		\end{enumerate}
	\end{enumerate}
\end{proof}

\begin{lemma}\label{Braidrelationsact}
	For all $s \neq t \in S$ we have $p(n_s, n_t) = p(n_t, n_s)$ in $\Aut(\Cbf_{\mathcal{B}_{\mathcal{F}}})$.
\end{lemma}
\begin{proof}
	At first we transform some elements by only using relations which are already known to hold in $\Aut(\Cbf_{\mathcal{B}_{\mathcal{F}}})$, i.e. using Lemmas \ref{Psacts} - \ref{Lemma: relation conjugate root groups}. We compute the following:
	\begin{enumerate}[label=(\alph*)]
		\item $u = 1$ and $n_1 = n_t$: Then 
		\begin{align*}
			n_s^{-1} n_t^{-1} n_s n_t n_t &= n_t n_t^{-1} n_s^{-1} n_t n_s (n_t^{-2})^{n_s} n_t^2 \\
			n_s^{-1} n_t^{-1} n_s^{-1} n_t n_s n_t n_t &=  n_t n_t^{-1} n_s^{-1} n_t^{-1} n_s n_t n_s (n_s^{-2})^{n_t n_s} n_t^2 \\
			n_s^{-1} n_t^{-1} n_s^{-1} n_t^{-1} n_s n_t n_s n_t n_t &= n_t n_t^{-1} n_s^{-1} n_t^{-1} n_s^{-1} n_t n_s n_t n_s (n_t^{-2})^{n_s n_t n_s} n_t^2
		\end{align*}
	
		\item $U_t \ni u_t \neq 1$ and $n_1 = n_t$: Then
		\allowdisplaybreaks
		\begin{align*}
			n_s^{-1} n_t^{-1} n_s n_t u_t n_t &= n_s^{-1} n_t^{-1} n_s \overline{u}_t n_s^{-1} n_s n_t b_t \\
			&= n_s^{-1} n_t n_t^{-2} u_t' n_s n_t b_t \\
			&= n_s^{-1} n_t u_t'' n_t n_t^{-2} n_t^{-1} n_s n_t b_t \\
			&= n_s^{-1} \overline{u_t''} n_s n_s^{-1} n_t n_s n_s^{-1} b_t' n_t^{-2} n_s n_s^{-1} n_t^{-1} n_s n_t b_t \\
			&= u_t''' n_t n_t^{-1} n_s^{-1} n_t n_s b_t'' n_s^{-1} n_t^{-1} n_s n_t b_t \\
			n_s^{-1} n_t^{-1} n_s^{-1} n_t n_s n_t u_t n_t &= n_s^{-1} n_t^{-1} n_s^{-1} n_t n_s \overline{u}_t n_s^{-1} n_t^{-1} n_t n_s n_t b_t \\
			&= n_s^{-1} n_t^{-1} n_s n_s^{-2} u_t' n_s n_s^{-1} n_t n_s n_t b_t \qquad \qquad (\text{note: } u_t' \in U_s) \\
			&= n_s^{-1} n_t^{-1} n_s u_t'' n_s n_s^{-2} n_s^{-1} n_t n_s n_t b_t \\
			&= n_s^{-1} n_t^{-1} \overline{u_t''} n_t n_s n_s^{-1} n_t^{-1} n_s b_s n_s^{-2} n_s^{-1} n_t n_s n_t b_t \\
			&= u_t''' n_t n_t^{-1} n_s^{-1} n_t^{-1} n_s n_t n_s n_s^{-1} n_t^{-1} b_s n_s^{-2} n_t n_s n_s^{-1} n_t^{-1} n_s^{-1} n_t n_s n_t b_t \\
			&= u_t''' n_t n_t^{-1} n_s^{-1} n_t^{-1} n_s n_t n_s  b_s' n_s^{-1} n_t^{-1} n_s^{-1} n_t n_s n_t b_t \\
			(n_s^{-1} n_t^{-1})^2 (n_s n_t)^2 u_t n_t &= (n_s^{-1} n_t^{-1})^2 n_s n_t n_s \overline{u}_t n_s^{-1} n_t^{-1} n_s^{-1} n_s n_t n_s n_t b_t \\
			&= n_s^{-1} n_t^{-1} n_s^{-1} n_t n_t^{-2} u_t' n_s n_t n_s n_t b_t \\
			&= n_s^{-1} n_t^{-1} n_s^{-1} n_t u_t'' n_t n_t^{-2} n_t^{-1} n_s n_t n_s n_t b_t \\
			&= n_s^{-1} n_t^{-1} n_s^{-1} \overline{u_t''} n_t n_s n_t n_s n_s^{-1} n_t^{-1} n_s^{-1} b_t' n_t^{-2} n_s n_t n_s (n_s^{-1} n_t^{-1})^2 (n_s n_t)^2 b_t \\
			&= n_s^{-1} n_t^{-1} n_s^{-1} \overline{u_t''} n_s n_t n_s n_s^{-1} n_t^{-1} n_s^{-1} (n_t n_s)^2 b_t'' (n_s^{-1} n_t^{-1})^2 (n_s n_t)^2  b_t \\
			&= u_t''' n_t n_t^{-1} n_s^{-1} n_t^{-1} n_s^{-1} (n_t n_s)^2 b_t'' (n_s^{-1} n_t^{-1})^2 (n_s n_t)^2  b_t \\
			&= u_t''' n_t (n_t^{-1} n_s^{-1})^2 (n_t n_s)^2 b_t'' (n_s^{-1} n_t^{-1})^2 (n_s n_t)^2  b_t
		\end{align*}
	\end{enumerate}
	Now we prove the claim. Let $s\neq t \in S, s_1, \ldots, s_k \in S, w\in \langle s, t \rangle$ be such that $s_1 \cdots s_k, ws_1 \cdots s_k$ are reduced and let $v_i \in U_{s_i}, u_i \in U_s \cup U_t$. We show by induction on $\ell(w)$ that $p(n_s, n_t)^{-1} p(n_t, n_s)$ fixes the equivalence class $[u_1, \ldots, u_n, v_1, \ldots, v_k]$, where $(u_1, \ldots, u_n, v_1, \ldots, v_k)$ is of type $(w, s_1, \ldots, s_k)$. For $\ell(w) = 0$ we have $l(p(s, t) s_1 \cdots s_k) = m_{st} +k$. Since $(1, \ldots, 1)$ of type $(p(s,t))$ is equivalent to $(1,\ldots, 1)$ of type $(p(t, s))$ it follows that
	\begin{align*}
		\omega( p(n_s, n_t), [v_1, \ldots, v_k] ) &= [1, \ldots, 1, v_1, \ldots, v_k] = \omega( p(n_t, n_s), [v_1, \ldots, v_k] )
	\end{align*}
	and hence $\omega( p(n_s, n_t)^{-1} p(n_t, n_s), [v_1, \ldots, v_k]) = [v_1, \ldots, v_k]$. Now let $\ell(w) >0$. Note that it suffices to show that one of $p(n_s, n_t)^{-1} p(n_t, n_s)$ and $p(n_t, n_s)^{-1} p(n_s, n_t)$ acts trivial, as the product of these two elements act trivial. Using the previous computations and induction, we infer:
	\begin{align*}
		\omega( p(n_s, n_t)^{-1} p(n_t, n_s), [(u_1, \ldots, u_n, v_1, \ldots, u_k)]) &= \omega( p(n_s, n_t)^{-1} p(n_t, n_s) u_1 n_1, [(u_2, \ldots, u_n, v_1, \ldots, u_k)]) \\
		&= [u_1, \ldots, u_n, v_1, \ldots, v_k]
	\end{align*}
	For example we consider the case $1 \neq u_1 \in U_s, n_1 = n_s$ and $m_{st} = 4$ explicitly. We have the following:
	\begin{align*}
		\omega((n_s n_t)^{-2} (n_t n_s)^2, [u_1, \ldots, u_k]) &= \omega((n_s n_t)^{-2} (n_t n_s)^2 u_1 n_s, [u_2, \ldots, u_k]) \\
		&= \omega( u_s''' n_s (n_s^{-1} n_t^{-1})^2 (n_s n_t)^2 b_s'' (n_t^{-1} n_s^{-1})^2 (n_t n_s)^2  b_s, [u_2, \ldots, u_k] ) \\
		&= \omega( u_s''' n_s  b_s''  b_s, [u_2, \ldots, u_k] )
	\end{align*}
	As before, we have $u_s''' = u_1$ and $b_s'' b_s = 1$ in $X_{s, t}$. This finishes the claim. \qedhere

\end{proof}

\begin{theorem}\label{welldefinedaction}
	The mapping $G \times \Cbf_{\mathcal{B}_{\mathcal{F}}} \to \Cbf_{\mathcal{B}_{\mathcal{F}}}, \left( g, [(u_1, \ldots, u_k)] \right) \mapsto \left[ \omega(g, \left( u_1, \ldots, u_k \right)) \right]$ is a well-defined action.
\end{theorem}
\begin{proof}
	This is a consequence of Theorem \ref{presentationofanRGDsystem} and the Lemmas \ref{Lemma: BJ homotopic sequences to homotopic sequences},  \ref{Lemma: ns automorphism} - \ref{Braidrelationsact}.
\end{proof}

\section{Main results}

\begin{remark}
	Let $\mathcal{F} := ( (\Delta_J)_{J \in E(S)}, (c_J)_{J \in E(S)}, (\phi_{rst})_{\{r, s\}, \{s, t\} \in E(S)})$ be an irreducible $3$-spherical Moufang foundation such that every panel contains at least $6$ chambers and such that each residue $\mathcal{F}_J$ of rank $3$ is integrable. Then $\mathcal{F}$ satisfies Condition $\lco$ and for each $J \subseteq S$ with $\vert J \vert = 3$ the $J$-residue $\mathcal{F}_J$ satisfies the Conditions $\lco$ and $\lsco$.
\end{remark}

\begin{theorem}\label{Theorem: rank 3 intrgrable implies RGD-system}
	Let $\mathcal{F}$ be an irreducible, $3$-spherical Moufang foundation of rank at least $3$ such that every panel contains at least $6$ chambers. Assume that for each irreducible $J \subseteq S$ with $\vert J \vert = 3$ the $J$-residue $\mathcal{F}_J$ is integrable. Then $\mathcal{D}_{\mathcal{F}}$ is an RGD-system and $\mathcal{F} \cong \mathcal{F}(\Delta(\mathcal{D}_{\mathcal{F}}), B_+)$. In particular, $\mathcal{F}$ is integrable.
\end{theorem}
\begin{proof}
	By Theorem \ref{welldefinedaction} we have an action of $G$ on $\Cbf_{\mathcal{B}_{\mathcal{F}}}$. Clearly, $U_s \to G$ is injective. Let $1\neq u' \in U_{-s}$. Then $1 \neq u := n_s^{-1} u' n_s \in U_s$ and $\omega(n_s u n_s^{-1}, [()]) = \omega(n_s, [(u)]) = [ \overline{u} ]$. Since $\overline{u} = 1$ if and only if $u=1$, the element $u'$ acts non-trivial on the building and hence $U_{-s} \to G$ is injective as well. Since for each $\alpha \in \Phi_+$ we have $U_{\alpha} \leq \langle U_{\alpha_s} \mid s\in S\rangle$ and $U_{\alpha_s}$ acts trivial on $[()]$, but $U_{-\alpha_s}$ does not fix $[()]$, $\mathcal{D}_{\mathcal{F}}$ satisfies (RGD$3$). Thus $\mathcal{D}_{\mathcal{F}}$ is an RGD-system of type $(W, S)$ and the claim follows from Corollary \ref{inclusionsinjectivergd3}.
\end{proof}

\begin{corollary}\label{Corollary: Main result}
	Let $\mathcal{F}$ be an irreducible, $3$-spherical Moufang foundation of rank at least $3$ such that every panel contains at least $6$ chambers. Then the following are equivalent:
	\begin{enumerate}[label=(\roman*)]
		\item $\mathcal{F}$ is integrable.
		
		\item For each irreducible $J \subseteq S$ with $\vert J \vert = 3$ the $J$-residue $\mathcal{F}_J$ is integrable.
	\end{enumerate}
\end{corollary}
\begin{proof}
	One implication follows directly by considering restriction of the twin building to the twin building $(R_J(c_+), R_J(c_-), \delta_*)$. The other follows from the previous theorem.
\end{proof}

\begin{theorem}\label{Theorem: Classification result}
	Let $\Delta$ be a thick irreducible $3$-spherical twin building of rank at least $3$. Then $\Delta$ is known.
\end{theorem}
\begin{proof}
	At first we assume that every panel of $\Delta$ contains at least $6$ chambers. Then $\Delta$ satisfies the Conditions $\lco$ and $\lsco$. Let $c$ be a chamber of $\Delta$ and let $\mathcal{F} := \mathcal{F}(\Delta, c)$. As in the previous corollary, each rank $3$-residue of $\mathcal{F}$ is integrable. Theorem \ref{Theorem: rank 3 intrgrable implies RGD-system} implies that $\mathcal{D}_{\mathcal{F}}$ is an RGD-system and $\mathcal{F} \cong \mathcal{F}(\Delta(\mathcal{D}_{\mathcal{F}}), B_+)$. Using Proposition \ref{Proposition: isomorphism of foundations yields iso of twin buildings}, we deduce $\Delta \cong \Delta(\mathcal{D}_{\mathcal{F}})$.
	
	Now we assume that there exists a panel containing at most $5$ chambers. Then \cite[$(34.5)$]{TW02} implies that every panel contains only finitely many chambers. But then $\Delta$ is known by the Main result of \cite{Mu99}. We note that the Main result of \cite{Mu99} uses the fact that no rank $2$ residue is associated with $B_2(2)$ in order the use the extension theorem of \cite{MR95}. But as it is shown in \cite[Corollary $6.4$]{BM23}, the extension theorem holds for arbitrary thick $3$-spherical twin buildings. Thus every irreducible locally finite $3$-spherical twin building is known by \cite{BM23} and \cite{Mu99}.
\end{proof}

\bibliography{references}
\bibliographystyle{abbrv}

\end{document}